\documentclass{article}

\usepackage{amsmath, amssymb, amsthm, mathtools, bm}
\usepackage[a4paper, margin=25mm]{geometry}
\usepackage[T1]{fontenc}
\usepackage[shortlabels]{enumitem}
\usepackage{xcolor}
\usepackage{xpatch}

\usepackage[
	sorting=nyt,
	maxbibnames=99,
	backref,
	backrefstyle=none, 
	style=alphabetic,
    maxalphanames=4, 
    minalphanames=1   
]{biblatex}

\usepackage[
    colorlinks,                
    linkcolor=blue!70!black,   
    citecolor=green!70!black,  
    urlcolor=cyan!70!black     
]{hyperref}
\usepackage[capitalise]{cleveref}

\DeclareFieldFormat*{title}{\mkbibemph{#1\isdot}}

\DeclareFieldFormat{url}{url: \href{#1}{\nolinkurl{#1}}}
\DeclareFieldFormat{doi}{doi: \href{https://doi.org/#1}{\nolinkurl{#1}}}
\DeclareFieldFormat{eprint:arxiv}{arXiv: \href{https://arxiv.org/abs/#1}{\nolinkurl{#1}}}
\DeclareFieldFormat{eprint}{arXiv: \href{https://arxiv.org/abs/#1}{\nolinkurl{#1}}}
\AtEveryBibitem{\clearfield{issn}\clearfield{isbn}\clearlist{language}}

\renewbibmacro*{doi+eprint+url}{%
    \iftoggle{bbx:eprint}{%
		\usebibmacro{eprint}%
	}{}
    \iffieldundef{eprint}{%
        \printfield{doi}%
    }{%
    }
	\newunit\newblock%
	\iffieldundef{doi}{%
    \iffieldundef{eprint}{%
		\usebibmacro{url+urldate}%
	}{}}{}%
}

\NewBibliographyString{toappear}
\DefineBibliographyStrings{english}{%
	toappear = {to appear},
}

\renewbibmacro*{in:}{%
	\iffieldundef{pubstate}
	{}
	{\printfield{pubstate}%
		\setunit{\addspace}%
		\clearfield{pubstate}}%
}

\DefineBibliographyStrings{english}{
    backrefpage={cited on p\adddot},
    backrefpages={cited on pp\adddot}
}

\DeclareFieldFormat{backrefparens}{\addperiod\addspace{\scriptsize(#1)}}
\xpatchbibmacro{pageref}{parens}{backrefparens}{}{}

\theoremstyle{plain}
\newtheorem{theorem}{Theorem}[section]
\newtheorem{lemma}[theorem]{Lemma}
\newtheorem{proposition}[theorem]{Proposition}
\newtheorem{corollary}[theorem]{Corollary}
\newtheorem{problem}[theorem]{Problem}

\theoremstyle{definition}
\newtheorem{remark}[theorem]{Remark}

\AddToHook{env/lemma/begin}{\crefalias{theorem}{lemma}}
\AddToHook{env/proposition/begin}{\crefalias{theorem}{proposition}}
\AddToHook{env/corollary/begin}{\crefalias{theorem}{corollary}}
\AddToHook{env/conjecture/begin}{\crefalias{theorem}{conjecture}}

\numberwithin{equation}{section}

\crefname{subsection}{subsection}{subsections}

\setlist[enumerate]{label=(\arabic*)}

\makeatletter
\newcommand{\step}[1]{%
  \par
  \addvspace{\medskipamount}
  \noindent%
  \textit{#1\@addpunct{.}}\enspace\ignorespaces
}
\makeatother

\makeatletter
\NewCommandCopy\latexparagraph\paragraph
\RenewDocumentCommand{\paragraph}{sO{#3}m}{%
  \IfBooleanTF{#1}
    {\latexparagraph*{\maybe@addperiod{#3}}}
    {\latexparagraph[#2]{\maybe@addperiod{#3}}}%
}
\newcommand{\maybe@addperiod}[1]{%
  #1\@addpunct{.}%
}
\makeatother

\DeclareMathOperator{\Ex}{\mathbb{E}}
\DeclareMathOperator{\Ent}{H}
\DeclareMathOperator{\mI}{I} 
\DeclareMathOperator{\supp}{supp} 

\newcommand{\defeq}{\coloneqq}

\let\epsilon\varepsilon

\newcommand{\RR}{\mathbb{R}}
\newcommand{\RRnn}{\mathbb{R}^+_0}
\newcommand{\RRp}{\mathbb{R}^+}
\newcommand{\NN}{\mathbb{N}}

\newcommand{\cB}{\mathcal{B}}
\newcommand{\cE}{\mathcal{E}}
\newcommand{\cL}{\mathcal{L}}
\newcommand{\cN}{\mathcal{N}}
\newcommand{\cP}{\mathcal{P}}
\newcommand{\cQ}{\mathcal{Q}}
\newcommand{\cS}{\mathcal{S}}
\newcommand{\cZ}{\mathcal{Z}}

\newcommand{\bone}{\mathbf{1}}

\newcommand{\ba}{\mathbf{a}}
\newcommand{\bfb}{\mathbf{b}}

\newcommand{\be}{\mathbf{e}}
\newcommand{\bfm}{\mathbf{m}}
\newcommand{\bn}{\mathbf{n}}
\newcommand{\br}{\mathbf{r}}
\newcommand{\bs}{\mathbf{s}}
\newcommand{\bu}{\mathbf{u}}
\newcommand{\bx}{\mathbf{x}}
\newcommand{\by}{\mathbf{y}}
\newcommand{\bz}{\mathbf{z}}

\newcommand{\balpha}{{\bm{\alpha}}}
\newcommand{\bbeta}{{\bm{\beta}}}
\newcommand{\blambda}{{\bm{\lambda}}}
\newcommand{\bmu}{{\bm{\mu}}}
\newcommand{\bfeta}{{\bm{\eta}}}
\newcommand{\bomega}{{\bm{\omega}}}
\newcommand{\bsigma}{{\bm{\sigma}}}
\newcommand{\btau}{{\bm{\tau}}}

\newcommand{\Zi}{\widetilde{Z}} 

\newcommand{\sym}{{\mathrm{sym}}}
\newcommand{\blank}{\makebox[1ex]{\(\cdot\)}}

\NewDocumentCommand{\D}{mm}{D(#1\mathbin{\Vert}#2)} 

\DeclarePairedDelimiter{\abs}{\lvert}{\rvert}
\DeclarePairedDelimiter{\norm}{\lVert}{\rVert}

\newcommand{\Mc}{{\upshape M}} 

\newcommand{\cLor}{Lorentzian} 

\newcommand{\footremember}[2]{%
    \footnote{#2}
    \newcounter{#1}
    \setcounter{#1}{\value{footnote}}%
}
\newcommand{\footrecall}[1]{%
    \footnotemark[\value{#1}]%
}

\begin{document}

\title{Antiferromagnetic models are clique-minimizing}
\author{Joonkyung Lee%
\footremember{Yonsei}{
    Department of Mathematics, Yonsei University, Seoul, South Korea. Research supported by Samsung STF Grant SSTF-BA2201-02 and the National Research Foundation of Korea (NRF) grant MSIT NRF-2022R1C1C1010300. Email: \texttt{\{joonkyunglee, jaehyeonseo\}@yonsei.ac.kr}.
}
\and 
Jaehyeon Seo\footrecall{Yonsei}}
\date{}

\maketitle

\begin{abstract}
An edge-weighted graph \(H\), possibly with loops, is \emph{antiferromagnetic} if its adjacency matrix is entrywise nonnegative and has at most one positive eigenvalue, counted with multiplicity. We show that, for any graph \(G\) with \(d_v\defeq \deg_G(v)\),
\[
    \hom(G,H) \ge \prod_{v\in V(G)} \hom(K_{d_v+1},H)^{\frac{1}{d_v+1}},
\]
whenever \(H\) is antiferromagnetic. In fact, we prove a vertex-inhomogeneous strengthening of this inequality, allowing a different fugacity vector at each vertex of \(G\). This gives a common generalization of the lower-bound inequalities of Sah, Sawhney, Stoner, and Zhao for independent sets, of Csikv\'{a}ri for \(q\)-colorings, and of the authors for semiproper colorings with at most two proper colors. Furthermore, it confirms recent conjectures of the authors and of Davies and LeBlanc. A key ingredient, of independent interest, is a strengthening of the delete-one form of Shearer’s inequality for Lorentzian measures, which provides a new approach to graph homomorphism inequalities.
\end{abstract}

\section{Introduction}

One of the fundamental reasons why complete graphs are important in graph theory is their repeated emergence as extremizers in counting problems on graphs. Throughout, let \(G\) be a finite simple graph, and write \(v(G)\defeq \abs{V(G)}\). For example, if \(i(G)\) denotes the number of independent sets of a \(d\)-regular graph \(G\), Cutler and Radcliffe \cite{cutler2014maximum} proved that
\[
    i(G)^{1/v(G)} \ge i(K_{d+1})^{1/(d+1)}.
\]
Thus, after the natural normalization that makes the quantity invariant under vertex-disjoint unions of copies of the same graph, the complete graph \(K_{d+1}\) minimizes the number of independent sets. Csikv\'{a}ri's theorem gives the analogous statement for the number of proper \(q\)-colorings; see \cite{zhao2017extremal}. Further work in this direction includes results for the independence polynomial and arbitrary degree sequences \cite{sah2019number}, partial results for the antiferromagnetic Ising model on cubic graphs \cite{davies2024occupancy}, and results for semiproper colorings with a small number of proper colors \cite{lee2026lower}. These results may suggest a common principle, but their proofs are tailored to the individual models and, until now, no general criterion explained when complete graphs should be minimizers.
This leads to a fundamental question:
\begin{center}
    \emph{What structural condition on a counting model forces complete graphs to be minimizers?}
\end{center}
We identify a broad class of models for which complete graphs are minimizers, specified by a simple spectral condition.
Let \(I\) be a finite nonempty set of spins. By a \emph{model} on \(I\), we mean a symmetric matrix \(H=(H(\sigma,\tau):\sigma,\tau\in I)\) with nonnegative entries.
Equivalently, a model \(H\) is an edge-weighted graph on~\(I\), possibly with loops.
For a graph \(G\), its weighted homomorphism count into \(H\) is
\[
    \hom(G,H)
    \defeq
    \sum_{\phi\colon V(G)\to I} \prod_{uv\in E(G)} H(\phi(u),\phi(v)).
\]
When the edge weights of \(H\) are \(0\)--\(1\)-valued, this is the usual number of graph homomorphisms from \(G\) to~\(H\). A symmetric entrywise-nonnegative matrix is \emph{antiferromagnetic} if it has at most one positive eigenvalue, counted with multiplicity; a model \(H\) is \emph{antiferromagnetic} if its underlying matrix is antiferromagnetic.
Our main result is the following. For a graph \(G\) and \(v\in V(G)\), write \(d_v\defeq \deg_G(v)\).

\begin{theorem}\label{antiferromagnetic-clique-minimizing}
Let \(H\) be an antiferromagnetic model. For every graph \(G\),
\begin{equation}\label{eq:clique-minimizing}
    \hom(G,H) \ge \prod_{v\in V(G)} \hom(K_{d_v+1},H)^{1/(d_v+1)}.
\end{equation}
\end{theorem}

Following the terminology of Sah, Sawhney, Stoner, and Zhao \cite{sah2020reverse}, we say that a model \(H\) is \emph{clique-minimizing} if \eqref{eq:clique-minimizing} holds for every graph \(G\). Thus, \Cref{antiferromagnetic-clique-minimizing} says that every antiferromagnetic model is clique-minimizing. In particular, if \(G\) is \(d\)-regular, the conclusion becomes
\[
    \hom(G,H)^{1/v(G)} \ge \hom(K_{d+1},H)^{1/(d+1)},
\]
while for an irregular graph the bound is a product of clique homomorphism counts determined by the individual vertex degrees.

The class of antiferromagnetic models contains the principal examples motivating the problem.  The hard-core model
\(
    K_2^\circ=\begin{psmallmatrix}1&1\\1&0\end{psmallmatrix}
\)
encodes independent sets; the choice \(H=K_q\) encodes proper \(q\)-colorings; and, more generally, \(K_q^{\ell\circ}\), the complete graph on \(q\) vertices with loops on exactly \(\ell\) vertices, encodes semiproper colorings. The antiferromagnetic Ising model with edge activity \(0\le B\le1\) has interaction matrix
\(
    H_B=\begin{psmallmatrix}B&1\\1&B\end{psmallmatrix}
\).
Each of these matrices has at most one positive eigenvalue. Consequently, \Cref{antiferromagnetic-clique-minimizing} recovers the known clique-minimization results for independent sets and proper colorings, extends the latter from regular graphs to arbitrary degree sequences, and establishes the corresponding inequality for all antiferromagnetic Ising models and (semi)proper colorings.  More generally, it settles our conjecture that every antiferromagnetic model is clique-minimizing \cite[Conjecture~1.3]{lee2026lower}.

\subsection{A vertex-inhomogeneous strengthening}

We in fact prove a stronger result in which every vertex of the source graph may have its own fugacity vector. Let \(G\) be a graph and \(H\) be a model on a spin set \(I\). Denote by \(\RRnn\) the set of nonnegative reals. For each \(v\in V(G)\), let \(\blambda^{(v)}=(\lambda^{(v)}_\sigma:\sigma\in I) \in (\RRnn)^I\) be a vector of vertex fugacities, and define the \emph{vertex-inhomogeneous partition function of \(H\)} by
\[
    \Zi_G(\blambda^{(v)}:v\in V(G))
    \defeq
    \sum_{\bsigma\in I^{V(G)}}
    \prod_{uv\in E(G)} H(\sigma_u,\sigma_v)
    \prod_{v\in V(G)} \lambda^{(v)}_{\sigma_v}.
\]
For a single fugacity vector \(\blambda\), define the \emph{(vertex-homogeneous) partition function} by
\[
    Z_G(\blambda)
    \defeq
    \Zi_G(\blambda:v\in V(G)),
\]
and for \(m\ge1\), let \(Z_m\defeq Z_{K_m}\) be a \emph{clique partition function}.
In particular, \(Z_G(\bone)=\hom(G,H)\), where \(\bone\) is the all-ones vector.

\begin{theorem}\label{antiferromagnetic-inhomo-clique-minimizing}
Let \(H\) be an antiferromagnetic model. For every graph \(G\) and every collection of fugacity vectors
\(
    \blambda^{(v)}\in(\RRnn)^I
\),
\begin{equation}\label{eq:inhomo-clique-minimizing}
    \Zi_G(\blambda^{(v)}:v\in V(G))
    \ge
    \prod_{v\in V(G)}
    Z_{d_v+1}(\blambda^{(v)})^{1/(d_v+1)}.
\end{equation}
\end{theorem}

Taking \(\blambda^{(v)}=\bone\) for every \(v\) gives \Cref{antiferromagnetic-clique-minimizing}. Moreover, equality in \eqref{eq:inhomo-clique-minimizing} holds whenever \(G\) is a disjoint union of cliques and the fugacity vectors are constant on each connected component, so the inequality is sharp.

We remark that \eqref{eq:inhomo-clique-minimizing} characterizes the class of antiferromagnetic models. Suppose that a model \(H\) satisfies \eqref{eq:inhomo-clique-minimizing} for every graph \(G\) and every collection of fugacity vectors. Let \(G=K_2\) and assign arbitrary fugacity vectors \(\bx,\by\in(\RRnn)^I\) to its two vertices. Since \(H\) is entrywise nonnegative, the resulting inequality is equivalent, after squaring, to
\[
    (\bx^\top H \by)^2 \ge (\bx^\top H \bx) (\by^\top H \by).
\]
It follows from \cite[Theorem~5.3]{choe2004homogeneous} that \(H\) is antiferromagnetic.
By contrast, the clique-minimizing inequality in \eqref{eq:clique-minimizing} does not characterize the class of antiferromagnetic models; in \Cref{sec:concluding-remark}, we will see examples of clique-minimizing models which are not antiferromagnetic.

The vertex-inhomogeneous formulation also provides a variety of examples. It captures multivariate independence polynomials, weighted list-homomorphism and list-coloring problems, and spin systems with vertex-dependent external fields.
It is also the natural setting for one-vertex marginals: logarithmic differentiation with respect to the fugacity of a specified spin at a specified vertex recovers the corresponding marginal probability, a viewpoint central to the occupancy method \cite{davies2017independent,davies2025hardcore}.
Related multivariate partition functions arise in Potts models with external fields \cite{ellis2011tutte} and in zero-freeness and cluster-expansion methods \cite{jenssen2024cluster,scott2005repulsive}.

As a consequence, \Cref{antiferromagnetic-inhomo-clique-minimizing} yields results that do not follow from the homogeneous formulation alone.
In particular, it settles a conjecture from our previous work \cite[Conjecture~4.2]{lee2026lower}.
It also extends the vertex-inhomogeneous results for \(K_2^\circ\) and \(K_3^\circ\) from our earlier work \cite{lee2026lower}, where \(K_q^\circ\defeq K_q^{1\circ}\), which strengthen \cite[Theorem~1.7]{sah2019number} by replacing a common fugacity with arbitrary vertex-dependent fugacities. 
Finally, the theorem confirms a conjecture of Davies and LeBlanc \cite[Conjecture~5(ii)]{davies2024occupancy} on antiferromagnetic Ising models and regular source graphs, which they described as ``somewhat bold''.

We finally remark that separate fugacity parameters also help in our inductive argument on $v(G)$ in the proof of~\Cref{antiferromagnetic-inhomo-clique-minimizing}. If a vertex \(w\) is deleted after conditioning on its spin~\(\sigma\), then the fugacity vector at each neighbor \(v\in N(w)\) is multiplied coordinatewise by the \(\sigma\)-row of \(H\). Hence, a vertex-homogeneous problem immediately becomes vertex-inhomogeneous under the basic recurrence for the partition function.

\subsection{Related graph homomorphism inequalities}

The lower bound in \Cref{antiferromagnetic-clique-minimizing} fits into a broader extremal picture in which cliques and complete bipartite graphs provide complementary extremizers.
The classical extremal theory for independent sets begins with work of Alon and Kahn and culminates, in the regular case, in Zhao's theorem that \(K_{d,d}\) maximizes the normalized hard-core partition function \cite{alon1991independent,kahn2001entropy,zhao2010number}. Galvin and Tetali extended Kahn's entropy method to weighted homomorphism counts on bipartite regular graphs \cite{galvin2004weighted}. To extend the comparison to irregular graphs, Sah, Sawhney, Stoner, and Zhao \cite{sah2020reverse} introduced the following notion: a model \(H\) is \emph{biclique-maximizing} if every graph \(G\) without isolated vertices satisfies
\begin{equation}\label{eq:biclique-maximizing}
    \hom(G,H)
    \le
    \prod_{uv\in E(G)}
    \hom(K_{d_u,d_v},H)^{1/(d_ud_v)}.
\end{equation}
They proved this inequality for several broad classes and conjectured that every antiferromagnetic model is biclique-maximizing \cite[Conjecture~1.16]{sah2020reverse}. This conjecture remains open.

On the other hand, in the same paper~\cite[Theorem~1.14]{sah2020reverse}, it is shown that every \emph{ferromagnetic} model, meaning every symmetric positive-semidefinite model, is \emph{clique-maximizing}:
\[
    \hom(G,H)
    \le
    \prod_{v\in V(G)}
    \hom(K_{d_v+1},H)^{1/(d_v+1)}.
\]
Thus, positive semidefiniteness forces the clique expression to be an upper bound, whereas \Cref{antiferromagnetic-clique-minimizing} shows that antiferromagnetism forces the same expression to be a lower bound. For antiferromagnetic models, the expected global picture is therefore that cliques and bicliques provide the sharp lower and upper bounds, respectively. The present paper establishes the lower half of this picture in full generality.

In our related work with Oh \cite{lee2025counting}, we used Lorentzian polynomials to study the biclique-maximizing conjecture from the perspective of the source graph. The present paper uses the same algebraic connection in a different direction: rather than proving an upper bound by comparison with bicliques, we extract concavity and entropy contraction properties of clique partition functions and use them to prove the lower bound.

\subsection{Methodology}\label{sec:methodology}

We develop a connection between Lorentzian polynomials and entropy contraction, yielding a new approach to graph homomorphism inequalities.

The first ingredient is the Lorentzian property of clique partition functions. Introduced by Br\"{a}nd\'{e}n and Huh \cite{branden2020lorentzian}, Lorentzian polynomials provide a robust framework for log-concavity and negative-dependence phenomena. Their relevance here stems from the fact that, in degree two, the Hessian condition in the definition of Lorentzian polynomials coincides with the antiferromagnetic condition. In particular, \Cref{H-antiferromagnetic-Z-Lorentzian} shows that every clique partition function of an entrywise-positive antiferromagnetic model is Lorentzian, a result that already appeared in our work with Oh \cite{lee2025counting}. Here, the resulting concavity provides the tangent-plane inequalities that serve as the starting point for our entropy argument.

The second ingredient is a strengthening of the delete-one case of Shearer's inequality \cite[Equation~(22)]{chung1986intersection} for exchangeable {\cLor} laws. Here, a law on \(A^m\), where \(A\) is a finite set of letters, is exchangeable if it is invariant under coordinate permutations. The count vector of a word records the multiplicity of each letter, and we call a law {\cLor} if the multivariate probability generating function of its count vector is Lorentzian. At the level of generating polynomials, this notion may be viewed as a multiset analogue of the Lorentzian probability measures introduced in \cite[Section~4.5]{branden2020lorentzian}.

For a random word \(X=(X_1,\ldots,X_m)\) with an exchangeable {\cLor} law, we show in \Cref{entropy-deletion} that
\[
    \sum_{j=1}^m \Ent(X_{[m]\setminus j})
    \ge (m-2)\Ent(X)+\sum_{j=1}^m\Ent(X_j).
\]
Indeed, subtracting the nonnegative ``total correlation'' \(\sum_{j=1}^m\Ent(X_j)-\Ent(X)\) from the lower bound gives the delete-one form of Shearer's inequality.

Our main tool in proving this is a one-coordinate relative entropy contraction, which may be viewed as a multiset analogue of entropic independence due to Anari, Jain, Koehler, Pham, and Vuong \cite{anari2022entropic}.
In \Cref{one-coord-relative-entropy-contraction}, we show that if \(\mu\) and \(\nu\) are exchangeable laws on \(A^m\) and \(\nu\) is {\cLor}, then
\[
    \D{\mu_1}{\nu_1} \le \frac{1}{m}\D{\mu}{\nu},
\]
where \(\mu_1\) and \(\nu_1\) are their one-coordinate marginals. Exchangeability identifies \(\D{\mu}{\nu}\) with the relative entropy between the induced count laws, and the contraction follows from concavity of the \(m\)-th root of a Lorentzian generating polynomial and the variational formula for relative entropy. Although entropy inequalities have long been used in extremal graph theory, this use of relative entropy contraction appears to be new in the study of graph homomorphism inequalities.

\subsection{Proof strategy}\label{sec:proof-strategy}

After reducing to entrywise-positive models, we follow the standard inductive localization strategy used in our previous work \cite{lee2026lower} as well as others~\cite{sah2019number,sah2020reverse}. First, we choose a vertex \(w\) of maximum degree $d_w=\Delta$, expand the vertex-inhomogeneous partition function according to the spin assigned to \(w\), and apply the induction hypothesis to \(G\setminus w\).

For instance, let $H=K_2^\circ$ and write $\blambda^{(v)}=(\lambda_v,\mu_v)$, where $\lambda_v$ and \(\mu_v\) are the fugacities of the non-looped and looped vertices of $H$, respectively. Then $Z_{d}(\blambda^{(v)})=\mu_v^{d} + d\lambda_v\mu_v^{d-1}$. Using a recurrence relation between $\Zi_G$ and $\Zi_{G\setminus w}$ that extends the standard recurrence relation for vertex-homogeneous partition functions
\[
    Z_G(\lambda,\mu)
    = \mu Z_{G\setminus w}(\lambda,\mu) + \lambda \mu^{d_w} Z_{G\setminus(\{w\}\cup N(w))}(\lambda,\mu),
\]
we see that the inductive strategy succeeds if we prove the inequality
\begin{equation}\label{eq:indep-multiaff-lower-bd-goal-2}
	\mu_w\prod_{v\in N(w)} Z_{d_v}(\blambda^{(v)})^{\frac{1}{d_v}}
		+ \lambda_w\prod_{v\in N(w)} Z_{d_v}(0,\mu_v)^{\frac{1}{d_v}}
	\ge Z_{\Delta+1}(\blambda^{(w)})^{\frac{1}{\Delta+1}}
		\prod_{v\in N(w)} Z_{d_v+1}(\blambda^{(v)})^{\frac{1}{d_v+1}}.
\end{equation}
This slightly generalizes Equation (2.2) in~\cite{lee2026lower}, which treats the case $\mu_v=1$ for every $v\in V(G)$. Once we focus on variables $\mu_w$ and $\lambda_w$,~\eqref{eq:indep-multiaff-lower-bd-goal-2} is simply of the form
\begin{align}\label{eq:simplified_tangent}
    A\mu_w + B\lambda_w \geq (\mu_w^{\Delta+1} + (\Delta+1) \lambda_w\mu_w^{\Delta})^{\frac{1}{\Delta+1}},
\end{align}
where the left-hand side is linear and the right-hand side is concave.
The first obstacle is that both \(A\) and~\(B\) are complicated products of rational powers of partition functions \(Z_{d_v}\) and \(Z_{d_v+1}\) over \(v\in N(w)\). In \cite{lee2026lower}, we resolved this issue by looking at the ``dual set'' \(\cS_\Delta\), which characterizes tangent (hyper)planes of \((Z_{\Delta+1})^{\frac{1}{\Delta+1}}\). The log-convexity of \(\cS_\Delta\) allowed us to handle the factors on each \(v\in N(w)\) individually, reducing the inequality to the \emph{local membership problem}, which asks whether a particular vector belongs to \(\cS_\Delta\).

This framework successfully carries over to a more general setting using the theory of Lorentzian polynomials. It is not hard to check that the dual set is again log-convex for any model (\Cref{SDelta-log-convex}). Also, since \(Z_{d}\) is Lorentzian, \((Z_{d})^{1/d}\) is concave and therefore its tangent hyperplane lies above its graph. However, the local membership problem remains challenging. Even in our previous work, \cite[Lemma~3.2]{lee2026lower} illustrates that the problems may be complicated even for the specific case \(H=K_3^\circ\). As a result, in~\cite{lee2026lower}, we verified them by model-specific, low-dimensional calculations only for \(K_2^\circ\) and \(K_3^\circ\). Our main technical contribution, culminating in \Cref{SDelta-membership}, solves this problem for general antiferromagnetic models by replacing the previous ad-hoc arguments with an abstract information-theoretic approach.

To illustrate, fix a neighbor \(v\) of \(w\), and let \(d\le\Delta\) be its degree. In~\Cref{sec:membership}, the desired local membership problem is shown to be equivalent to the statement that an auxiliary quantity, which we call the \emph{optimized pressure}, is nonpositive (\Cref{F-eq-log-M}). Entropy inequalities for the associated Lorentzian laws, combined with the Gibbs variational formula, yield the monotonicity relation for the optimized pressure established in \Cref{abstract-pressure-monotonicity}.
Concavity of \((Z_{d+1})^{1/(d+1)}\) establishes nonpositivity at clique size \(d+1\) (\Cref{antiferromagnetic-Mm-bound}), and monotonicity extends this conclusion to every clique size up to \(\Delta+1\).
This concludes the proof of \Cref{SDelta-membership}.

\paragraph{Organization}
The remainder of the paper is organized as follows. \Cref{sec:preliminaries} fixes notation and collects preliminary lemmas. It also reviews Lorentzian polynomials and relates them to antiferromagnetic models via clique partition functions. \Cref{sec:entropy} develops entropy inequalities for exchangeable {\cLor} laws. In \Cref{sec:abstract-optimized-pressure}, we combine these results to derive a monotonicity property of the optimized pressure. We apply this abstract framework to prove \Cref{antiferromagnetic-inhomo-clique-minimizing} in \Cref{sec:antiferromagnetic}. Finally, \Cref{sec:concluding-remark} concludes with further remarks.

\section{Preliminaries}\label{sec:preliminaries}

Throughout this section, \(J\) denotes a finite index set. For a positive integer \(m\), write \([m]\defeq\{1,\ldots,m\}\). Recall that \(\RRnn\) is the set of nonnegative real numbers; we also write \(\RRp\) for the set of positive reals and \(\NN_0\) for the set of nonnegative integers. For a function \(f\) of variables \((x_\sigma)_{\sigma\in J}\), write \(\partial_\sigma f \defeq \partial_{x_\sigma} f\). We denote by \(\bone\) the all-ones vector of the relevant dimension.

We generally use boldface letters for vectors and corresponding ordinary letters for their entries. For example, if \(A\) is a set and \(\bx\in A^J\), we write \(\bx=(x_i : i\in J)\). We identify \(\bx\) with the function \(J\to A\), \(i\mapsto x_i\). For \(\bx,\by\in\RR^J\), write \(\bx\odot\by \defeq (x_i y_i : i\in J)\) for the coordinatewise product. For \(\balpha\in(\NN_0)^J\), let \(\bx^{\balpha} \defeq \prod_{i\in J} x_i^{\alpha_i}\).

The following is an elementary result for homogeneous functions.

\begin{proposition}[Euler's homogeneous function theorem]\label{euler-homogeneous}
Let \(f\) be a real-valued function on an open cone \(C\subseteq\RR^J\). Suppose \(f\) is differentiable and homogeneous of degree \(k\in\RR\), i.e., \(f(t\bx) = t^k f(\bx)\) for all \(t>0\). Then
\[
    \sum_{i\in J} x_i \partial_i f(\bx) = k f(\bx).
\]
\end{proposition}
\begin{proof}
Fix \(\bx\in C\) and differentiate the identity \(f(t\bx)=t^k f(\bx)\) with respect to \(t\) at \(t=1\).
\end{proof}

We will use the following form of the Perron--Frobenius theorem; see \cite[Section~7]{meyer2023matrix}.
Let \(B=(B_{ij})_{i,j\in J}\) be a square matrix with nonnegative entries.
We say \(B\) is \emph{reducible} if there exists a permutation matrix \(P\) such that
\[
    PBP^\top = \begin{pmatrix}
        X & Y \\ 0 & Z
    \end{pmatrix},
\]
where \(X\) and \(Z\) are square matrices of order at least \(1\). If no such permutation matrix exists, \(B\) is \emph{irreducible}. If \(B\) is symmetric, then there is a permutation matrix \(P\) such that
\[
    PBP^\top =
    \begin{pmatrix}
        B_1 & 0 & \cdots & 0\\
        0 & B_2 & \cdots & 0\\
        \vdots & \vdots & \ddots & \vdots\\
        0 & 0 & \cdots & B_r
    \end{pmatrix},
\]
where each diagonal block \(B_i\) is irreducible.
We call the matrices \(B_i\) the \emph{irreducible blocks} of \(B\).

\begin{theorem}[Perron--Frobenius]
\label{Perron--Frobenius}
Let \(B\) be a nonzero irreducible square matrix with nonnegative entries. Then its spectral radius \(\rho(B)\) is a positive eigenvalue of \(B\), and \(B\) has a corresponding entrywise-positive eigenvector.
\end{theorem}

\begin{corollary}\label{antiferromagnetic-irreducible}
Let \(B\) be an antiferromagnetic matrix with no zero rows. Then \(B\) is irreducible.
\end{corollary}
\begin{proof}
Since \(B\) is symmetric, the decomposition above is block diagonal. Because \(B\) has no zero rows, every irreducible block of \(B\) is nonzero. By \Cref{Perron--Frobenius}, each irreducible block of \(B\) has at least one positive eigenvalue, which is also an eigenvalue of \(B\). Since \(B\) is antiferromagnetic, it has only one irreducible block and is therefore irreducible.
\end{proof}

\medskip

We recall the recursive definition of Lorentzian polynomials from \cite{branden2020lorentzian}. Let \(I\) be a finite nonempty index set. For a real polynomial \(f\) in the variables \(x_\sigma\), \(\sigma\in I\), its \emph{Hessian} \(\nabla^2 f\) is the matrix whose rows and columns are indexed by \(I\) and whose \((\sigma,\tau)\)-entry is \(\partial_\sigma\partial_\tau f\). The \emph{support} of \(f\) is the set of \(\balpha\in (\NN_0)^I\) such that \(\bx^{\balpha}\) appears in \(f\) with nonzero coefficient. A set \(S\subseteq (\NN_0)^I\) is \emph{\Mc-convex} if, for every \(\balpha,\bbeta\in S\), whenever \(\alpha_\sigma>\beta_\sigma\) for some \(\sigma\in I\), there is \(\tau\in I\) such that
\[
    \alpha_\tau<\beta_\tau,
    \qquad
    \balpha-\mathbf{e}_\sigma + \mathbf{e}_\tau \in S,
    \qquad
    \bbeta-\mathbf{e}_\tau + \mathbf{e}_\sigma \in S,
\]
where \(\mathbf{e}_\sigma\) is the standard unit vector whose \(\sigma\)-coordinate is \(1\) and whose other coordinates are \(0\).
A polynomial \(f\in \RRnn[x_\sigma : \sigma\in I]\) of degree \(d\) is \emph{Lorentzian} if either \(f\) is a constant or \(d\ge1\) and the following conditions hold:
\begin{enumerate}
\item \(f\) is homogeneous and has \Mc-convex support;
\item if \(d=2\), \(\nabla^2 f\) is antiferromagnetic;
\item if \(d\ge3\), for each \(\sigma\in I\), \(\partial_\sigma f\) is either zero or Lorentzian.
\end{enumerate}

\begin{proposition}\label{Lorentzian-implies-root-concavity}
If \(f\) is a Lorentzian polynomial of degree \(d\ge1\), then \(f^{1/d}\) is concave on the positive orthant.
\end{proposition}
\begin{proof}
For \(d=1\), this is immediate because \(f\) is linear. For \(d\ge2\), this follows from Theorem~2.30 and Proposition~2.33 in \cite{branden2020lorentzian}.
\end{proof}

\begin{proposition}\label{Lorentzian-closure}
Let $f(\bx)$ be an $n$-variable Lorentzian polynomial of degree $d$.
Then the following polynomials are also Lorentzian:
\begin{enumerate}
    \item $cf(\bx)$ for a positive scalar $c$;
    \item $f(\ba\odot\bx)$ for any $\ba\in(\RRp)^n$;
    \item $D_{\bfb} f(\bx)\defeq \sum_{i=1}^n b_i\partial_i f(\bx)$ for any $\bfb\in(\RRnn)^n$.
\end{enumerate}
\end{proposition}
\begin{proof}
The first assertion is immediate from the definition, and the remaining assertions follow from Theorem~2.10 and Corollary~2.11 in \cite{branden2020lorentzian}.
\end{proof}

We need an auxiliary lemma on derivatives of clique partition functions, which is useful in connecting the theory of Lorentzian polynomials to our setting. 

\begin{lemma}\label{partition-function-derivative}
Let \(H\) be any model on a spin set \(I\) and \(Z_d\) be its clique partition function.
For \(d\ge1\), \(\bx\in(\RRp)^I\), and \(\tau\in I\),
\begin{equation}\label{eq:partition-function-derivative}
    \partial_\tau Z_{d+1}(\bx) = (d+1) Z_d(\bfm_\tau\odot\bx),
\end{equation}
where \(\bfm_\tau\) is the \(\tau\)-row of \(H\).
\end{lemma}
\begin{proof}
Applying \(\partial_\tau\) to \(Z_{d+1}\) amounts to choosing a vertex in \(K_{d+1}\) and fixing its spin to be \(\tau\). Identifying \(V(K_{d+1})\) with \([d+1]\), we obtain
\begin{align*}
    \partial_\tau Z_{d+1}(\bx)
    &= \sum_{w\in [d+1]}
    \sum_{\substack{\bsigma\in I^{[d+1]}\\ \sigma_w=\tau}}
    \prod_{uv\in \binom{[d+1]}{2}} H(\sigma_u, \sigma_v)
    \prod_{v\in [d+1]\setminus \{w\}} x_{\sigma_v}
    \\&= (d+1) \sum_{\bsigma\in I^{[d]}}
    \prod_{uv\in \binom{[d]}{2}} H(\sigma_u,\sigma_v)
    \prod_{v\in [d]} ( H(\tau,\sigma_v) \, x_{\sigma_v} )
    \\&= (d+1) \, Z_{d}(\bfm_\tau\odot\bx).
    \qedhere
\end{align*}
\end{proof}

The following lemma is a key statement that allows us to use the theory of Lorentzian polynomials.

\begin{lemma}\label{H-antiferromagnetic-Z-Lorentzian}
Let \(H\) be an entrywise-positive antiferromagnetic model. For \(d\ge1\), the clique partition function \(Z_d\) of \(H\) is Lorentzian. Consequently, \(Z_d^{1/d}\) is concave on the positive orthant.
\end{lemma}
\begin{proof}
Let \(I\) be the spin set of \(H\). We prove that \(Z_d\) is Lorentzian by induction on \(d\); the final assertion then follows from \Cref{Lorentzian-implies-root-concavity}.
Clearly, \(Z_d\) is homogeneous of degree \(d\). Since \(H\) is entrywise positive, the support of \(Z_d\) is the full degree-\(d\) simplex \(\{\balpha\in (\NN_0)^I : \sum_{\sigma\in I} \alpha_\sigma = d\}\), which is readily seen to be \Mc-convex.
When \(d=1\), homogeneity and the \Mc-convexity of the support show that \(Z_d\) is Lorentzian; when \(d=2\), \(\nabla^2 Z_d = 2H\) is antiferromagnetic, so \(Z_d\) is Lorentzian.

Now assume \(d\ge3\). Fix \(\tau\in I\).
By \eqref{eq:partition-function-derivative} with \(d\) replaced by \(d-1\), \(\partial_\tau Z_d(\bx) = d Z_{d-1}(\bfm_\tau\odot\bx)\).
Since \(Z_{d-1}\) is Lorentzian by the inductive hypothesis and \(\bfm_\tau\) is entrywise positive, \Cref{Lorentzian-closure} shows that \(\partial_\tau Z_d\) is Lorentzian. Since this holds for every \(\tau\in I\), \(Z_d\) is Lorentzian.
\end{proof}

A stronger version without the entrywise positivity hypothesis was proved in \cite[Theorem~3.1]{lee2025counting} using an extra lemma to show \Mc-convexity in full generality.
For the present paper, \Cref{H-antiferromagnetic-Z-Lorentzian} suffices because, in \Cref{sec:antiferromagnetic}, we will reduce the proof of \Cref{antiferromagnetic-inhomo-clique-minimizing} to the entrywise-positive case.

\section{Entropy inequalities for exchangeable {\cLor} laws}
\label{sec:entropy}

Throughout this and subsequent sections, all random variables take values in finite sets. All logarithms are natural, and we use the convention \(0\log0 = 0\log(0/0) = 0\). By a \emph{law}, we mean a probability measure on a finite set.

Let \(\Omega\) be a finite nonempty set. Let \(X=(X_1,\dots,X_m)\) be an \(\Omega^m\)-valued random vector. For \(J\subseteq [m]\) and \(j\in J\), let \(X_{J}\defeq (X_i : i\in J)\) and \(X_{J\setminus j} \defeq X_{J\setminus\{j\}}\). Given a law \(\mu\) on \(\Omega^m\) and a random vector \((X_1,\dots,X_m)\sim\mu\), write \(\mu_i\defeq\cL(X_i)\) for the \(i\)-th marginal of \(\mu\), \(i\in[m]\).

We first recall standard entropy identities and variational formulae; for more background, see \cite{cover2006elements}. In \Cref{exchangeable-{\cLor}-laws-and-entropy-contraction}, we apply them to exchangeable {\cLor} laws to derive our entropy inequalities.

\subsection{Entropy and relative entropy}
\label{sec:entropy-and-relative-entropy}

Let \(\mu\) be a law on \(\Omega\). The \emph{entropy} of \(\mu\) is
\[
    \Ent(\mu) \defeq - \sum_{\omega\in\Omega} \mu(\omega) \log \mu(\omega).
\]
We use the elementary fact that entropy is concave: if \(\nu_1,\ldots,\nu_n\) are laws on \(\Omega\) and \(p_1,\ldots,p_n\ge0\) with \(\sum_i p_i=1\), then
\[
    \Ent \biggl(\,\sum_{i=1}^n p_i\nu_i \biggr) \ge \sum_{i=1}^n p_i \Ent(\nu_i).
\]
This follows from concavity of the function \(t\mapsto -t\log t\).

Let \(X\) and \(Y\) be random variables. The entropy \(\Ent(X)\) of \(X\) is defined as \(\Ent(\cL(X))\). Let
\[
    \Ent(X\mid Y) \defeq \sum_{y:\Pr(Y=y)>0} \Pr(Y=y) \Ent(X\mid Y=y),
\]
where \(\Ent(X\mid Y=y)\) denotes the entropy of the conditional law \(\cL(X\mid Y=y)\). Equivalently, \(\Ent(X\mid Y)=\Ent(X,Y)-\Ent(Y)\), where \(\Ent(X,Y)\) is the entropy of the joint law of \(X\) and \(Y\).
The \emph{mutual information} between \(X\) and \(Y\) is
\[
    \mI(X;Y) \defeq \Ent(X) + \Ent(Y) - \Ent(X,Y)
    = \Ent(X) - \Ent(X\mid Y).
\]

Let \(\mu\) and \(\nu\) be laws on \(\Omega\). Suppose \(\mu\ll\nu\), i.e., \(\mu(\omega)=0\) for all \(\omega\in\Omega\) with \(\nu(\omega)=0\). The \emph{relative entropy} \(\D{\mu}{\nu}\) is
\begin{align}\label{eq:variational}
     \D{\mu}{\nu}
     \defeq \sum_{\omega\in\Omega} \mu(\omega) \log \frac{\mu(\omega)}{\nu(\omega)}
     = \sup_{f\colon\Omega\to\RR} \bigl\{ \Ex_\mu f - \log \Ex_\nu e^f \bigr\}.
\end{align}
The supremum representation is the \emph{variational formula for relative entropy}, also known as the Donsker--Varadhan variational formula; see \cite[Appendix~C.2]{dupuis1997weak}. Relative entropy is nonnegative, with equality if and only if \(\mu=\nu\). If \(\mu\not\ll\nu\), we use the convention \(\D{\mu}{\nu}=+\infty\).

\begin{lemma}[Chain rule for relative entropy]\label{relative-entropy-chain-rule}
Let \(\mu\) and \(\nu\) be laws on \(\Omega\). Let \(\Xi\) be a finite set and let \(T\colon\Omega\to\Xi\) be a map. Assume \(\mu\ll\nu\). Let \(\mu'\) and \(\nu'\) be the pushforward laws on \(\Xi\), i.e.,
\[
    \mu'(\xi) \defeq \Pr_{X\sim\mu} (T(X)=\xi),
    \qquad
    \nu'(\xi) \defeq \Pr_{X\sim\nu} (T(X)=\xi).
\]
Then
\[
    \D{\mu}{\nu} = \D{\mu'}{\nu'}
    + \Ex_{Y\sim \mu'} \D{\mu(\blank\!\mid T=Y)}{\nu(\blank\!\mid T=Y)}.
\]
\end{lemma}
\begin{proof}
Let \(\Xi_+\defeq\{\xi\in\Xi : \mu'(\xi)>0\}\). We then rewrite $\D{\mu}{\nu}$ as
\begin{align}\label{eq:rephrase}
    \D{\mu}{\nu}
    = \sum_{\xi\in\Xi_+} \sum_{\omega\in T^{-1}(\xi)} \mu(\omega) \log\frac{\mu(\omega)}{\nu(\omega)}.
\end{align}
As \(\mu\ll\nu\), \(\nu'(\xi)>0\) holds for every \(\xi\in\Xi_+\). Hence, for \(\xi\in\Xi_+\) and \(\omega\in T^{-1}(\xi)\),
\[
    \mu(\omega)=\mu'(\xi) \mu(\omega\mid T=\xi),
    \qquad
    \nu(\omega)=\nu'(\xi) \nu(\omega\mid T=\xi).
\]
Substituting these into~\eqref{eq:rephrase} gives
\begin{align*}
    \D{\mu}{\nu}
    &= \sum_{\xi\in\Xi_+} \sum_{\omega\in T^{-1}(\xi)} \mu(\omega) \log\frac{\mu'(\xi)}{\nu'(\xi)}
        + \sum_{\xi\in\Xi_+} \sum_{\omega\in T^{-1}(\xi)} \mu(\omega)
    \log \frac{\mu(\omega\mid T=\xi)}{\nu(\omega\mid T=\xi)}
    \\&= \D{\mu'}{\nu'}
        + \sum_{\xi\in\Xi_+} \mu'(\xi) \D{\mu(\blank\!\mid T=\xi)}{\nu(\blank\!\mid T=\xi)}.
    \qedhere
\end{align*}
\end{proof}

\begin{proposition}[Mutual information as relative entropy]
\label{mutual-info-formula}
For random variables \(X\) and \(Y\),
\[
    \mI(X;Y)
    = \Ex_Y [ \D{\cL(X\mid Y)}{\cL(X)} ]
    = \sum_{y:\Pr(Y=y)>0} \Pr(Y=y) \D{\cL(X\mid Y=y)}{\cL(X)}.
\]
\end{proposition}
\begin{proof}
Let \(\mu\defeq \cL(X)\), and for each \(y\) with \(\Pr(Y=y)>0\), let \(\mu_y\defeq \cL(X\mid Y=y)\). Since
\[
    \mu(x) = \sum_{y:\Pr(Y=y)>0} \Pr(Y=y) \mu_y(x),
\]
we have \(\mu_y\ll\mu\) whenever \(\Pr(Y=y)>0\). Therefore,
\begin{align*}
    \mI(X;Y)
    &= \Ent(X) - \Ent(X\mid Y)
    \\&= -\sum_x \mu(x) \log \mu(x) + \sum_{y:\Pr(Y=y)>0} \Pr(Y=y) \sum_x \mu_y(x) \log \mu_y(x)
    \\&= \sum_{y:\Pr(Y=y)>0} \Pr(Y=y) \biggl( \sum_x \mu_y(x) \log \mu_y(x) - \sum_x \mu_y(x) \log \mu(x) \biggr)
    \\&= \sum_{y:\Pr(Y=y)>0} \Pr(Y=y) \D{\mu_y}{\mu}
    \\&= \Ex_Y [ \D{\cL(X\mid Y)}{\cL(X)} ].
    \qedhere
\end{align*}
\end{proof}

The \emph{Gibbs law} associated with \(f\colon\Omega\to\RR\) is the law \(\mu_f(\omega)\defeq e^{f(\omega)}/Z_f\), where \(Z_f \defeq \sum_{\omega\in\Omega} e^{f(\omega)}\). Combining this with the variational formula~\eqref{eq:variational} for relative entropy gives the \emph{Gibbs variational formula}.

\begin{proposition}[Gibbs variational formula]\label{Gibbs-variational-formula}
    Let \(f\colon\Omega\to\RR\). Then
    \[
        \log \sum_{\omega\in\Omega} e^{f(\omega)}
        = \sup_{\mu} \{ \Ex_\mu f + \Ent(\mu)\},
    \]
    where the supremum is over all laws on \(\Omega\). The unique optimizer is the Gibbs law \(\mu_f\).
\end{proposition}
\begin{proof}
    Let \(\mu\) be a law on \(\Omega\). Since \(\log\mu_f(\omega) = f(\omega) - \log Z_f\), we have
    \[
        \D{\mu}{\mu_f}
        = \sum_{\omega\in\Omega} \mu(\omega) \log \frac{\mu(\omega)}{\mu_f(\omega)}
        = \sum_{\omega\in\Omega} \mu(\omega) \log\mu(\omega) - \sum_{\omega\in\Omega} \mu(\omega) \log\mu_f(\omega)
        = -\Ent(\mu) - \Ex_\mu f + \log Z_f.
    \]
    Since \(\D{\mu}{\mu_f}\ge0\), rearranging gives \(\Ex_\mu f + \Ent(\mu) \le \log Z_f\), with equality if and only if \(\mu=\mu_f\). Taking the supremum proves the formula and the uniqueness of the optimizer.
\end{proof}

\subsection{Exchangeable {\cLor} laws and entropy contraction}
\label{exchangeable-{\cLor}-laws-and-entropy-contraction}

In this subsection, we introduce {\cLor} laws and establish properties that yield a strengthening of a special case of Shearer's inequality \cite[Equation~(22)]{chung1986intersection}; see also \cite[Proposition~15.7.4]{alon2016probabilistic}. As mentioned in \Cref{sec:methodology}, {\cLor} laws provide a multiset analogue of the Lorentzian probability measures introduced in \cite[Section~4.5]{branden2020lorentzian}. This is achieved by allowing letters to occur with multiplicity in each word \(\bomega\) in the sample space. The Lorentzian property of their count-generating polynomials, which will be defined below, allows us to prove a relative entropy contraction (\Cref{one-coord-relative-entropy-contraction}). This yields the desired entropy inequality in \Cref{entropy-deletion}.

Fix a finite set of letters \(A\). For \(m\ge1\), \(\bomega\in A^m\), and \(a\in A\), let $N_a(\bomega)$ denote the frequency of the letter $a$ in $\bomega$, i.e., $N_a(\bomega)\defeq \abs{\{j\in [m] : \omega_j=a\}}$. Let
\[
    T(\bomega) \defeq (N_a(\bomega) : a\in A)
\]
be the \emph{count vector} of \(\bomega\), which counts frequencies of letters in $A$. For $\bomega\in A^m$, all the frequencies sum to $m$. That is, $T(\bomega)\in \cQ_m(A)$, where \(\cQ_m(A) \defeq \{\bn\in (\NN_0)^A : \sum_{a\in A} n_a = m\}\).
For a law \(\mu\) on \(A^m\), define its \emph{count law} on \(\cQ_m(A)\) by
\[
    \pi_\mu(\bn) \defeq \Pr_{X\sim\mu} (T(X) = \bn),
    \qquad \bn\in \cQ_m(A).
\]
For \(\bx\in\RR^A\), we define the multivariate probability generating function of \(\pi_\mu\) by
\[
    h_\mu(\bx)
    \defeq \Ex_{Z\sim\pi_\mu} \bx^Z
    = \Ex_{X\sim\mu} \bx^{T(X)}.
\]
Since \(T(X)\in\cQ_m(A)\), this is a homogeneous polynomial of degree \(m\), which we call the \emph{count-generating polynomial} of \(\mu\).
We say that \(\mu\) is \emph{{\cLor}} if \(h_\mu\) is Lorentzian.

A subset \(\Omega_m\subseteq A^m\) is called \emph{permutation-invariant} if it is invariant under permutations of the coordinate positions, i.e., \(\bomega\in\Omega_m\) implies \((\omega_{\pi(1)},\dots,\omega_{\pi(m)}) \in \Omega_m\) for every permutation \(\pi\) of \([m]\). Let \(\Omega_\bullet=(\Omega_m : m\ge1)\) be a support family. We say that \(\Omega_\bullet\) is permutation-invariant if \(\Omega_m\) is permutation-invariant for all \(m\), and \emph{deletion-compatible} if, for each \(m\ge2\),
\[
    \bomega\in\Omega_m
    \implies
    \bomega_{[m]\setminus j}\in\Omega_{m-1}
    \qquad\forall j\in[m].
\]
A law on \(A^m\) is \emph{exchangeable} if it is invariant under permutations of the coordinate positions.
Let \(\cN_m(\Omega_\bullet)\) be the family of exchangeable {\cLor} laws supported on \(\Omega_m\). For the remainder of this section, we fix a support family \(\Omega_\bullet\) which is permutation-invariant and deletion-compatible, and simply write \(\cN_m\defeq\cN_m(\Omega_\bullet)\) for \(m\ge1\). We identify a law on \(\Omega_m\) with the law on \(A^m\) obtained by assigning mass zero to \(A^m\setminus\Omega_m\).

Using the closure properties of Lorentzian polynomials in \Cref{Lorentzian-closure}, we obtain the following properties of exchangeable {\cLor} laws.

\begin{lemma}\label{Nm-closure}
Let \(m\ge2\), \(\mu\in\cN_m\), and \(X\sim\mu\). Then for \(j\in[m]\),
\begin{enumerate}[(1)]
    \item \label{Nm-closure-deletion} the marginal law of \(X_{[m]\setminus j}\) belongs to \(\cN_{m-1}\);
    \item \label{Nm-closure-conditioning} for each \(a\in A\), if \(\Pr(X_j=a)>0\), the conditional law \(\cL(X_{[m]\setminus j} \mid X_j=a)\) belongs to \(\cN_{m-1}\).
\end{enumerate}
\end{lemma}
\begin{proof}
By exchangeability of \(\mu\), it suffices to prove the statement for \(j=m\).

\step{Proof of \ref{Nm-closure-deletion}}
Let \(\nu \defeq \cL(X_{[m-1]})\), which is exchangeable on \(A^{m-1}\) since \(\mu\) is exchangeable. For \(a\in A\), let \(\be_a\in(\NN_0)^A\) be the standard basis vector whose \(a\)-coordinate is \(1\). If the event \(\{X_m=a\}\) occurs, the count vector of \(X_{[m-1]}\) is \(T(X)-\be_a\). Moreover, exchangeability of \(\mu\) implies that, for \(\bn\in\cQ_m(A)\) with \(\pi_\mu(\bn)>0\), the probability that $X_m=a$ conditioned on $T(X)=\bn$ is proportional to the number of occurrences of $a$ in $\bn$, i.e.,
\begin{equation}\label{eq:conditional-prob-Xm=a}
    \Pr_{X\sim\mu} (X_m=a \mid T(X)=\bn) = \frac{n_a}{m}.
\end{equation}
Thus, for \(\bx=(x_a:a\in A)\), the count-generating polynomial of \(\nu\) is
\begin{align*}
    h_{\nu}(\bx)
    &= \Ex_{Z\sim\pi_\nu} \bx^{Z}
    \\&= \sum_{\bn\in\cQ_m(A) : \pi_\mu(\bn)>0} \biggl( \pi_\mu(\bn) \sum_{a\in A : n_a>0} \Pr_{X\sim\mu} (X_m=a \mid T(X)=\bn) \, \bx^{\bn-\be_a} \biggr)
    \\&= \sum_{\bn\in\cQ_m(A) : \pi_\mu(\bn)>0} \biggl( \pi_\mu(\bn) \sum_{a\in A: n_a>0} \frac{n_a}{m} \bx^{\bn-\be_a} \biggr).
\end{align*}
If $n_a>0$, then $n_a \bx^{\bn-\be_a}= \partial_a \bx^{\bn}$; otherwise, $\partial_a \bx^{\bn}=0$, so $\sum_{a\in A: n_a>0} n_a \bx^{\bn-\be_a} = \sum_{a\in A}\partial_a \bx^{\bn}$. Thus, by exchanging the order of the sum above, we have $h_{\nu}(\bx)=\frac{1}{m} \sum_{a\in A} \partial_{a} h_\mu(\bx)$.
Since \(h_\mu\) is Lorentzian and \(h_\nu\) is its directional derivative in the direction \(\frac{1}{m} \bone\), \Cref{Lorentzian-closure} shows that \(h_\nu\) is Lorentzian. Hence, \(\nu\) is {\cLor}.
As \(\Omega_\bullet\) is deletion-compatible, \(\supp\nu \subseteq \Omega_{m-1}\) follows immediately. Therefore, \(\nu\in\cN_{m-1}\).

\step{Proof of \ref{Nm-closure-conditioning}}
Fix \(a\in A\) with \(\Pr(X_m=a)>0\). Let \(\rho\defeq\cL(X_{[m-1]} \mid X_m=a)\), which is indeed exchangeable on \(A^{m-1}\). 
Rewrite the count-generating polynomial of \(\rho\) as
\begin{align*}
    h_\rho(\bx)
    = \Ex_{X\sim\mu} \bigl[ \bx^{T(X)-\be_{a}}\mid X_m=a \bigr]
    = \frac{1}{\Pr(X_m=a)} \sum_{\bomega\in A^m : \omega_m=a} \Pr_{X\sim\mu} (X=\bomega) \, \bx^{T(\bomega)-\be_{a}}.
\end{align*}
\Cref{eq:conditional-prob-Xm=a} again gives
\begin{align*}
    \sum_{\bomega\in A^m : \omega_m=a} \Pr_{X\sim\mu} (X=\bomega) \, \bx^{T(\bomega)-\be_{a}}
    &= \sum_{\substack{\bn\in\cQ_m(A) :\\ n_a>0,\, \pi_\mu(\bn)>0}} \pi_\mu(\bn) \Pr_{X\sim\mu} (X_m=a\mid T(X)=\bn) \, \bx^{\bn-\be_{a}}
    \\&= \sum_{\substack{\bn\in\cQ_m(A) :\\ n_a>0,\, \pi_\mu(\bn)>0}} \pi_\mu(\bn) \frac{n_a}m \bx^{\bn-\be_{a}}
    \\&= \frac{1}{m} \partial_{a} h_\mu(\bx).
\end{align*}
Since \(h_\mu\) is Lorentzian, \Cref{Lorentzian-closure} again shows that \(h_\rho\) is Lorentzian. Thus, \(\rho\) is {\cLor}.
As \(\supp\rho \subseteq \supp\nu \subseteq \Omega_{m-1}\), we conclude that \(\rho\in\cN_{m-1}\).
\end{proof}

Next, we prove a one-coordinate relative entropy contraction for exchangeable laws \(\mu\) and \(\nu\), with \(\nu\) {\cLor}.

\begin{lemma}\label{one-coord-relative-entropy-contraction}
Let \(m\ge1\) and let \(\mu\) and \(\nu\) be exchangeable laws on \(A^m\). If \(\nu\) is {\cLor}, then
\[
    \D{\mu_1}{\nu_1} \le \frac{1}{m} \D{\mu}{\nu}.
\]
\end{lemma}
\begin{proof}
If \(\mu\not\ll\nu\), then \(\D{\mu}{\nu}=+\infty\), which trivializes the desired inequality. We thus assume \(\mu\ll\nu\).

We start by checking \(\mu_1\ll\nu_1\) and \(\pi_\mu\ll\pi_\nu\).
If \(\nu_1(a)=0\) for some \(a\in A\), then since \(\nu\) is exchangeable, \(\nu(\bomega)=0\) for every \(\bomega\in A^m\) that contains \(a\). From \(\mu\ll\nu\), it follows that \(\mu(\bomega)=0\) for all such \(\bomega\), whence \(\mu_1(a)=0\).
Similarly, if \(\pi_\nu(\bn)=0\) for some \(\bn\in\cQ_m(A)\), then \(\nu(\bomega)=0\) for all \(\bomega\in A^m\) whose count vector is \(\bn\), so again \(\mu(\bomega)=0\) for all such \(\bomega\in A^m\). Thus, \(\pi_\mu(\bn)=0\).

For \(a\in A\), let
\[
    u_a\defeq \Ex_{X\sim\mu} N_a(X) = \Ex_{Z\sim\pi_\mu} Z(a),
    \qquad v_a\defeq \Ex_{Y\sim\nu} N_a(Y) = \Ex_{W\sim\pi_\nu} W(a).
\]
In other words, $u_a$ and $v_a$ denote the expected numbers of occurrences of the letter $a$ in $X\sim\mu$ and $Y\sim \nu$, respectively.
Since \(\mu\) is exchangeable,
\[
    \mu_1(a) = \frac{1}{m} \sum_{j=1}^m \Pr_{X\sim\mu} (X_j=a)
    = \frac{1}{m} \Ex_{X\sim\mu} N_a(X)
    = \frac{u_a}{m},
\]
and similarly, \(\nu_1(a) = v_a/m\).

We claim that \(\D{\mu}{\nu} = \D{\pi_\mu}{\pi_\nu}\). Applying the chain rule for relative entropy,~\Cref{relative-entropy-chain-rule}, gives
\begin{equation}\label{eq:chain-rule-for-relative-entropy-applied}
    \D{\mu}{\nu} = \D{\pi_\mu}{\pi_\nu} + \Ex_{Z\sim\pi_\mu} \D{\mu(\blank\!\mid T=Z)}{\nu(\blank\!\mid T=Z)}.
\end{equation}
Since \(\pi_\mu\ll\pi_\nu\), the conditional law \(\nu(\blank\!\mid T=\bn)\) is well-defined whenever \(\pi_\mu(\bn)>0\). For every such \(\bn\), exchangeability implies that both \(\mu(\blank\!\mid T=\bn)\) and \(\nu(\blank\!\mid T=\bn)\) are uniform on \(\{\bomega\in A^m : T(\bomega)=\bn\}\). Hence, the second term on the right-hand side of \eqref{eq:chain-rule-for-relative-entropy-applied} vanishes, so \(\D{\mu}{\nu}=\D{\pi_\mu}{\pi_\nu}\). It remains to show
\begin{equation}\label{eq:one-coord-relative-entropy-contraction-reduced}
    \D{\mu_1}{\nu_1} \le \frac{1}{m} \D{\pi_\mu}{\pi_\nu}.
\end{equation}

Let \(h_\nu\) be the count-generating polynomial of \(\nu\). By definition, \(h_\nu(\bone)=1\), and for \(a\in A\), \(\partial_{a} h_\nu(\bone) = \Ex_{X\sim\nu} N_a(X) = v_a\). Thus,
\[
    \partial_{a} (h_\nu^{1/m}) (\bone) = \frac{1}{m} h_\nu(\bone)^{1/m-1} \partial_{a} h_\nu(\bone) = \frac{1}{m} v_a.
\]
Since \(\nu\) is {\cLor}, \(h_\nu\) is Lorentzian, so by \Cref{Lorentzian-implies-root-concavity}, \(h_\nu^{1/m}\) is concave on the positive orthant. Thus, for \(\bx\in(\RRp)^A\),
\[
    h_\nu(\bx)^{1/m} \le h_\nu(\bone)^{1/m} + \nabla (h_\nu^{1/m})(\bone) \cdot (\bx-\bone)
    = 1 + \frac{1}{m} \sum_{a\in A} v_a (x_a-1)
    = \frac{1}{m} \sum_{a\in A} v_a x_a,
\]
where the last equality uses \(\sum_{a\in A} v_a/m = \sum_{a\in A} \nu_1(a) = 1\).
From the definition of \(h_\nu\), it follows that
\begin{equation}\label{eq:h-nu-1/m-concavity-result}
    \Ex_{Z\sim\pi_\nu} \prod_{a\in A} x_a^{Z(a)}
    \le \biggl( \frac{1}{m} \sum_{a\in A} v_a x_a \biggr)^{\!m}.
\end{equation}

Now, for any \(\bfeta\in\RR^A\), the variational formula for relative entropy with the function \(f\colon \cQ_m(A)\to\RR\) defined by \(f(\bn)\defeq \sum_{a\in A} n_a \eta_a\) gives
\begin{align*}
    \D{\pi_\mu}{\pi_\nu} &\ge \Ex_{\pi_\mu} f - \log \Ex_{\pi_\nu} e^f
    \\&= \sum_{a\in A} \Ex_{Z\sim\pi_\mu} Z(a) \eta_a - \log \Ex_{Z\sim\pi_\nu} \exp\biggl(\, \sum_{a\in A} Z(a) \eta_a \biggr)
    \\&= \sum_{a\in A} u_a \eta_a - \log \Ex_{Z\sim\pi_\nu} \exp\biggl(\, \sum_{a\in A} Z(a) \eta_a \biggr).
\end{align*}
The expectation inside the logarithm is upper-bounded using \eqref{eq:h-nu-1/m-concavity-result} applied with \(x_a=e^{\eta_a}\) for \(a\in A\), so
\begin{equation}\label{eq:D-pi-mu-pi-nu-lower-bd}
    \D{\pi_\mu}{\pi_\nu}
    \ge \sum_{a\in A} u_a \eta_a - m \log \biggl( \frac{1}{m} \sum_{a\in A} v_a e^{\eta_a} \biggr).
\end{equation}

Finally, we compute \(\D{\mu_1}{\nu_1}\). Again, by the variational formula for relative entropy~\eqref{eq:variational},
\[
    \D{\mu_1}{\nu_1}
    = \sup_{f\colon A\to \RR} \{\Ex_{\mu_1} f - \log \Ex_{\nu_1} e^f\}.
\]
After reparametrizing the supremum by \((f(a):a\in A)\in\RR^A\), the right-hand side becomes
\[
    \sup_{\bfeta\in\RR^A} \biggl\{ \sum_{a\in A} \mu_1(a) \eta_a - \log \biggl(\, \sum_{a\in A} \nu_1(a) e^{\eta_a} \biggr) \biggr\}
    = \sup_{\bfeta\in\RR^A} \biggl\{ \frac{1}{m} \sum_{a\in A} u_a \eta_a - \log \biggl( \frac{1}{m} \sum_{a\in A} v_a e^{\eta_a} \biggr) \biggr\}
    \le \frac{1}{m} \D{\pi_\mu}{\pi_\nu},
\]
where the inequality follows from \eqref{eq:D-pi-mu-pi-nu-lower-bd}. This proves \eqref{eq:one-coord-relative-entropy-contraction-reduced} and completes the proof.
\end{proof}

\begin{remark}
\Cref{one-coord-relative-entropy-contraction} can also be obtained by using entropic independence~\cite[Theorem~4]{anari2022entropic} through polarization.
To sketch briefly, our count-generating functions \(h_\mu\) and \(h_\nu\) can be multilinearized by replacing each \(a\in A\) with \(m\) labeled copies $a_1,a_2,\dots,a_m$ with a suitable normalization to give laws \(\widetilde{\mu}\) and \(\widetilde{\nu}\) on \(\binom{A\times[m]}{m}\). For instance, $x_a^k$ is replaced by $\binom{m}{k}^{-1}\sum_{K\in\binom{[m]}{k}}\prod_{i\in K}x_{a_i}$.
Since this polarization preserves the Lorentzian property \cite[Proposition~3.1]{branden2020lorentzian}, the generating polynomial of \(\widetilde{\nu}\) is Lorentzian and hence log-concave, so \(\widetilde{\nu}\) is entropically independent. Because \(\mu\) and \(\nu\) are exchangeable, these lifts satisfy \(\D{\widetilde{\mu}}{\widetilde{\nu}} = \D{\mu}{\nu}\).
The one-element down marginals, obtained by choosing a uniformly random element of the lifted \(m\)-set, are \(\mu_1\otimes u_m\) and \(\nu_1\otimes u_m\), respectively, where \(u_m\) is the uniform law on \([m]\). The \(\frac{1}{m}\)-contraction for this operation therefore gives the lemma.
\end{remark}

Finally, we prove a strengthening of a special case of Shearer's inequality for exchangeable {\cLor} laws.
For a random vector \(X=(X_1,\dots,X_m)\), Shearer's inequality states that if \(\mathcal{G}\) is a family of subsets of \([m]\) such that each \(i\in[m]\) belongs to at least \(k\) members of \(\mathcal{G}\), then
\[
    \sum_{G\in\mathcal{G}} \Ent(X_i : i\in G) - k \Ent(X) \ge 0.
\]
When \(\mathcal{G}\) consists of the \((m-1)\)-subsets of \([m]\), the inequality becomes
\begin{equation}\label{eq:delete-one-Han}
    \sum_{j=1}^m \Ent(X_{[m]\setminus j}) - (m-1)\Ent(X) \ge 0.
\end{equation}
This is also called the delete-one case of Han's inequality \cite{han1978nonnegative}, which predates Shearer's inequality; see also \cite[Chapter~17]{cover2006elements}.
We strengthen this for \(X\sim\mu\in\cN_m\) by showing that
\[
    \sum_{j=1}^m \Ent(X_{[m]\setminus j}) - (m-1)\Ent(X)
    \ge \sum_{j=1}^m \Ent(X_j) - \Ent(X);
\]
the right-hand side is nonnegative by subadditivity of entropy.

\begin{lemma}\label{entropy-deletion}
    Let \(m\ge2\). For any \(X\sim \mu\in \cN_m\),
    \[
        \sum_{j=1}^m \Ent(X_{[m]\setminus j}) \ge (m-2) \Ent(X) + \sum_{j=1}^m \Ent(X_j).
    \]
\end{lemma}
\begin{proof}
We proceed by induction on \(m\).
The base case \(m=2\) holds with equality.
Assume \(m\ge3\) and let
\begin{align}\label{eq:delta_m}
    \Delta_m(X)
    \defeq \sum_{j=1}^m \Ent(X_{[m]\setminus j}) - (m-2) \Ent(X) - \sum_{j=1}^m \Ent(X_j),
\end{align}
so that our goal is to show $\Delta_m(X)\geq 0$.
Write \(Y\defeq X_{[m-1]}\) and \(Z\defeq X_m\).
Then $\sum_{j=1}^m \Ent(X_{[m]\setminus j}) = \Ent(Y) + \sum_{i=1}^{m-1} \Ent(Y_{[m-1]\setminus i},Z)$ and $\sum_{j=1}^m \Ent(X_j) = \Ent(Z) + \sum_{i=1}^{m-1} \Ent(Y_i)$.
Substituting these into~\eqref{eq:delta_m} gives
\begin{align*}
    \Delta_m(X)
   &= \Ent(Y) + \sum_{i=1}^{m-1} \Ent(Y_{[m-1]\setminus i},Z) - (m-2) \Ent(Y,Z) - \Ent(Z) - \sum_{i=1}^{m-1} \Ent(Y_i)
    \\&= \Ent(Y) + \sum_{i=1}^{m-1} \Ent(Y_{[m-1]\setminus i} \mid Z) - (m-2) \Ent(Y \mid Z) - \sum_{i=1}^{m-1} \Ent(Y_i).
\end{align*}
Using \(\Ent(Y) = \Ent(Y\mid Z) + \mI(Y;Z)\) and \(\Ent(Y_i) = \Ent(Y_i\mid Z) + \mI(Y_i;Z)\), we get
\begin{align*}
    \Delta_m(X) &= \biggl(\, \sum_{i=1}^{m-1} \Ent(Y_{[m-1]\setminus i} \mid Z) - (m-3) \Ent(Y\mid Z) - \sum_{i=1}^{m-1} \Ent(Y_i\mid Z) \biggr) + \mI(Y;Z) - \sum_{i=1}^{m-1} \mI(Y_i;Z)
    \\&= \Ex_Z \left[ \Delta_{m-1}(Y\mid Z)\right] + \mI(Y;Z) - \sum_{i=1}^{m-1} \mI(Y_i;Z).
\end{align*}
By \Cref{Nm-closure}, for every \(z\) with \(\Pr(Z=z)>0\), the conditional law \(\cL(Y\mid Z=z)\) belongs to \(\cN_{m-1}\). Hence, the induction hypothesis gives
\[
    \Ex_Z \Delta_{m-1}(Y\mid Z)\ge0,
\]
so it remains to show
\begin{equation}\label{eq:mutual-info-sum-ineq}
    \sum_{i=1}^{m-1} \mI(Y_i;Z) \le \mI(Y;Z).
\end{equation}
Let \(\nu\defeq \cL(Y)\), and for every \(z\) with \(\Pr(Z=z)>0\), let \(\rho_z\defeq \cL(Y\mid Z=z)\). \Cref{Nm-closure} implies that both \(\nu\) and \(\rho_z\) belong to \(\cN_{m-1}\).
\Cref{mutual-info-formula} and exchangeability of \(\nu\) and \(\rho_z\) give
\[
    \sum_{i=1}^{m-1} \mI(Y_i;Z)
    = \sum_{i=1}^{m-1} \Ex_Z [ \D{(\rho_Z)_i}{\nu_i} ]
    = (m-1) \Ex_Z [ \D{(\rho_Z)_1}{\nu_1} ].
\]
For every such \(z\), \Cref{one-coord-relative-entropy-contraction} implies \(\D{(\rho_z)_1}{\nu_1} \le \frac{1}{m-1} \D{\rho_z}{\nu}\). Applying \Cref{mutual-info-formula} again gives \(\mI(Y;Z) = \Ex_Z [\D{\rho_Z}{\nu}]\).
Combining these yields \eqref{eq:mutual-info-sum-ineq}, completing the proof.
\end{proof}

\begin{remark}
Since each \(X_{[m]\setminus j}\) is a function of \(X\), we have \(\Ent(X_{[m]\setminus j}) \le \Ent(X)\).
Using this with \Cref{entropy-deletion} implies that for every finite set of letters \(A\), if \(m\ge2\) and \(X\sim\mu\in\cN_m\),
\begin{equation}\label{eq:factor-two-entropy-inequality}
    \Ent(X)\geq \frac{1}{2}\sum_{j=1}^m \Ent(X_j).
\end{equation}
When \(A=\{0,1\}\), identify \(X\) with the random subset \(S_X\defeq\{j\in[m]:X_j=1\}\), and let \(\rho\) be the law of \(S_X\).
Following Anari--Oveis Gharan--Vinzant~\cite{anari2021logconcave}, a law \(\theta\) on \(2^{[m]}\) is called \emph{log-concave} if its multiaffine generating polynomial \(g_\theta(\bz)\defeq\sum_{S\subseteq[m]}\theta(S)\prod_{i\in S}z_i\) is log-concave on the positive orthant.
In the present exchangeable Lorentzian setting, polarization shows that both \(\rho\) and its complementary law \(\rho^\ast(S)\defeq\rho([m]\setminus S)\) are log-concave.
Hence, \eqref{eq:factor-two-entropy-inequality} also follows from \cite[Corollary~5.6]{anari2021logconcave}.
Thus, under our stronger exchangeable Lorentzian hypothesis, \Cref{entropy-deletion} refines \eqref{eq:factor-two-entropy-inequality} to a delete-one entropy inequality applicable to an arbitrary finite set of letters.
\end{remark}

\section{Optimized pressure and its monotonicity}
\label{sec:abstract-optimized-pressure}

We now convert the entropy estimates from \Cref{sec:entropy} into an abstract monotonicity principle for extremal logarithmic values of partition functions, which we call the optimized pressure\footnote{This terminology is motivated by the statistical-mechanical interpretation of the logarithm of a partition function, up to normalization, as a pressure; see, e.g., \cite[Chapter~3]{friedli2018statistical}.}.

Let \(A\) be a finite set of letters. Let \(\Omega_\bullet=(\Omega_m : m\ge1)\) be a support family which is permutation-invariant and deletion-compatible. Assume each \(\Omega_m\) is nonempty.
For each \(m\ge1\), let \(\cL_m\) be the family of laws on \(\Omega_m\). 
Denote by \(\cE_m\) and \(\cN_m\defeq\cN_m(\Omega_\bullet)\) the families of exchangeable laws and exchangeable {\cLor} laws on \(\Omega_m\), respectively. In particular, \(\cN_m\subseteq \cE_m\subseteq \cL_m\). Whenever \(\mu\in\cL_m\), write \(X=(X_1,\dots,X_m)\sim\mu\).

For each \(m\ge 1\), let \(U_m\colon \Omega_m\to\RR\) be a permutation-invariant energy function, meaning that \(U_m(\omega_{\pi(1)},\dots,\omega_{\pi(m)}) = U_m(\bomega)\) for every \(\bomega\in\Omega_m\) and every permutation \(\pi\) of \([m]\). Let \(\cP^+(A)\) be the family of positive laws on \(A\), i.e.,
\[
    \cP^+(A) \defeq \biggl\{\alpha\colon A\to\RRp : \sum_{a\in A} \alpha(a) = 1\biggr\}.
\]
For \(\alpha\in\cP^+(A)\), define
\[
    \cZ_{m}(\alpha) \defeq
    \sum_{\bomega\in\Omega_m}
    \exp(U_{m}(\bomega))
    \prod_{j=1}^m\alpha(\omega_j),
\]
and the \emph{optimized pressure}
\[
    F_m \defeq \sup_{\alpha\in\cP^+(A)} \log \cZ_{m}(\alpha).
\]
For \(\alpha\in\cP^+(A)\), define a law \(\mu_{m}^\alpha\) on \(\Omega_m\) by
\[
    \mu_{m}^\alpha(\bomega)
    \defeq
    \frac{1}{\cZ_{m}(\alpha)}
    \exp(U_{m}(\bomega))
    \prod_{j=1}^m\alpha(\omega_j),
\]
which is exchangeable since \(\Omega_m\) and \(U_m\) are permutation-invariant.

The following properties are possessed by some particular choices of $U_m$. We first derive consequences of these properties and later verify them for our specific choice of $U_m$. More precisely, in \Cref{sec:optimized-pressure}, we verify \ref{cond:mu-alpha-in-cN} for antiferromagnetic models using the Lorentzian property of the partition function and verify \ref{cond:energy-deletion} by a simple counting argument.
\begin{enumerate}[(P1)]
\item For \(m\ge1\) and \(\alpha\in\cP^+(A)\), \(\mu_{m}^\alpha\in \cN_m\).
\label{cond:mu-alpha-in-cN}

\item The energy function satisfies the deletion identity: for \(m\ge2\) and \(\bomega\in\Omega_m\),
\[
    \sum_{j=1}^m U_{m-1}(\bomega_{[m]\setminus j}) = (m-2)U_{m}(\bomega).
\]
\label{cond:energy-deletion}
\end{enumerate}

\begin{lemma}\label{abstract-Gibbs-variational-formula}
Assume \ref{cond:mu-alpha-in-cN}. Then for \(m\ge1\),
\begin{equation}\label{eq:abstract-Gibbs-variational-formula}
    F_m = \sup_{\mu\in\cN_m}
    \biggl\{ \Ex_\mu U_{m}(X) + \Ent(X) - \sum_{j=1}^m \Ent(X_j) \biggr\}.
\end{equation}
\end{lemma}
\begin{proof}
First, we represent \(F_m\) as a supremum taken over \(\cL_m\). Let \(\alpha\in\cP^+(A)\). Then applying the Gibbs variational formula,~\Cref{Gibbs-variational-formula}, to
\[
    f^\alpha(\bomega) \defeq U_m(\bomega) + \sum_{j=1}^m \log \alpha(\omega_j)
\]
gives
\[
    \log \sum_{\bomega\in\Omega_m} \exp(U_m(\bomega)) \prod_{j=1}^m \alpha(\omega_j)
    = \sup_{\mu\in\cL_m} \biggl\{ \Ex_\mu U_m(X) + \sum_{j=1}^m \Ex_\mu \log \alpha(X_j) + \Ent(X) \biggr\}.
\]
After taking the supremum over \(\alpha\in\cP^+(A)\), the left-hand side equals \(F_m\) by definition, so
\begin{equation}\label{eq:abstract-Gibbs-variational-formula-Lm}
    F_m = \sup_{\mu\in\cL_m}
    \biggl\{ \Ex_\mu U_{m}(X) + \sup_{\alpha\in\cP^+(A)} \sum_{j=1}^m \Ex_\mu \log \alpha(X_j) + \Ent(X) \biggr\}.
\end{equation}

Second, we prove the analogue of \eqref{eq:abstract-Gibbs-variational-formula} with \(\cN_m\) replaced by \(\cE_m\). For \(\mu\in\cL_m\), let \(\bar{\mu}_1\) be its average one-site marginal:
\[
    \bar{\mu}_1(a) \defeq \frac{1}{m} \sum_{j=1}^m \Pr_{X\sim\mu} (X_j=a),
    \qquad a\in A.
\]
Then, for \(\alpha\in\cP^+(A)\),
\[
    \sum_{j=1}^m \Ex_\mu \log \alpha(X_j) = m \sum_{a\in A} \bar{\mu}_1(a) \log \alpha(a).
\]
Taking the supremum over \(\alpha\in\cP^+(A)\) gives
\begin{equation}\label{eq:sup-Pplus-sum-expectation-log}
    \sup_{\alpha\in\cP^+(A)} \sum_{j=1}^m \Ex_\mu \log \alpha(X_j)
    = m \sup_{\alpha\in\cP^+(A)} \sum_{a\in A} \bar{\mu}_1(a) \log \alpha(a)
    = m \sum_{a\in A} \bar{\mu}_1(a) \log \bar{\mu}_1(a)
    = -m \Ent(\bar{\mu}_1).
\end{equation}
To justify the second equality, note that \(\bar{\mu}_1 \ll \alpha\) for every \(\alpha\in\cP^+(A)\). Then, nonnegativity of relative entropy yields
\[
    \sum_{a\in A} \bar{\mu}_1(a) \log \bar{\mu}_1(a) - \sum_{a\in A} \bar{\mu}_1(a) \log \alpha(a)
    = \D{\bar{\mu}_1}{\alpha} \ge 0,
\]
where equality occurs if \(\alpha=\bar{\mu}_1\in \cP^+(A)\). Thus, if \(\bar{\mu}_1\in\cP^+(A)\),
\begin{equation}\label{eq:sup-Pplus-bar-mu}
    \sup_{\alpha\in\cP^+(A)} \sum_{a\in A} \bar{\mu}_1(a) \log \alpha(a)
    = \sum_{a\in A} \bar{\mu}_1(a) \log \bar{\mu}_1(a).
\end{equation}
Even if \(\bar{\mu}_1\notin\cP^+(A)\) in general, let \(u\) be the uniform law on \(A\) and consider \(\alpha_\epsilon \defeq (1-\epsilon) \bar{\mu}_1 + \epsilon u \in \cP^+(A)\) for \(0<\epsilon<1\). Then
\[
    \sum_{a\in A} \bar{\mu}_1(a)\log\alpha_\epsilon(a)
    \to
    \sum_{a\in A} \bar{\mu}_1(a)\log \bar{\mu}_1(a)
    \quad\text{as}\quad
    \epsilon\to0,
\]
with the convention \(0\log0=0\). This implies \eqref{eq:sup-Pplus-bar-mu}, whence \eqref{eq:sup-Pplus-sum-expectation-log} follows.
Substituting \eqref{eq:sup-Pplus-sum-expectation-log} into \eqref{eq:abstract-Gibbs-variational-formula-Lm} gives
\begin{equation}\label{eq:abstract-Gibbs-variational-formula-Lm-2}
    F_m = \sup_{\mu\in\cL_m} \{ \Ex_\mu U_{m}(X) + \Ent(X) - m \Ent(\bar{\mu}_1) \}.
\end{equation}

For \(\mu\in\cL_m\), let \(\mu^\sym\) be the symmetrization of \(\mu\) over all permutations of the \(m\) sites, i.e.,
\[
    \mu^\sym \defeq \frac{1}{m!} \sum_{\pi\in\mathfrak{S}_m} \mu^{(\pi)},
    \qquad
    \mu^{(\pi)}(\omega_1,\dots,\omega_m) \defeq \mu(\omega_{\pi(1)},\dots,\omega_{\pi(m)}),
\]
where \(\mathfrak{S}_m\) is the symmetric group on \([m]\).
Then \(\mu^\sym\in\cE_m\), and permutation invariance of  \(U_{m}\) gives \(\Ex_{\mu^\sym} U_{m} = \Ex_\mu U_{m}\). By concavity of entropy,
\[
    \Ent(\mu^{\sym})
    \ge \frac{1}{m!} \sum_{\pi\in\mathfrak{S}_m}\Ent(\mu^{(\pi)})
    = \Ent(\mu).
\]
Moreover, \((\mu^\sym)_1 = \bar{\mu}_1\), so \(\Ent((\mu^\sym)_1) = \Ent(\bar{\mu}_1)\). Thus, symmetrization does not decrease the objective in \eqref{eq:abstract-Gibbs-variational-formula-Lm-2}. If \(\mu\) is exchangeable and \(X\sim\mu\), then \(\Ent(\bar{\mu}_1) = \Ent(X_j)\) for all \(j\in[m]\). Hence,
\begin{equation}\label{eq:abstract-Gibbs-variational-formula-Em}
    F_m = \sup_{\mu\in\cE_m} \biggl\{ \Ex_\mu U_{m}(X) + \Ent(X) - \sum_{j=1}^m \Ent(X_j) \biggr\},
\end{equation}
which is the desired analogue of \eqref{eq:abstract-Gibbs-variational-formula}.

Finally, we show \eqref{eq:abstract-Gibbs-variational-formula}.
Fix \(\alpha\in\cP^+(A)\) and let $X\sim\mu_m^\alpha$. Since \(\mu_m^\alpha\) is the Gibbs law associated with~\(f^\alpha\), the Gibbs variational formula for \(f^\alpha\), i.e.,~\Cref{Gibbs-variational-formula} with $f=f^\alpha$, gives
\begin{align*}
    &\Ex_{\mu_m^\alpha} \biggl( U_{m}(X) + \sum_{j=1}^m \log \alpha(X_j) \biggr) + \Ent(\mu_m^\alpha)
    \\&\qquad = \log\sum_{\bomega\in\Omega_m}e^{f^{\alpha}(\bomega)} 
    = \log \sum_{\bomega\in\Omega_m} \exp(U_{m}(\bomega)) \prod_{j=1}^m \alpha(\omega_j)
    = \log \cZ_{m}(\alpha).
\end{align*}
Adding \(\sum_{j=1}^m \D{(\mu_m^\alpha)_j}{\alpha}\) to both sides and using \(\D{(\mu_m^\alpha)_j}{\alpha} = \Ex_{\mu_m^\alpha} \log\frac{(\mu_m^\alpha)_j(X_j)}{\alpha(X_j)}\) gives
\[
    \Ex_{\mu_m^\alpha} U_{m}(X) + \Ent(\mu_m^\alpha) - \sum_{j=1}^m \Ent((\mu_m^\alpha)_j)
    = \log \cZ_{m}(\alpha) + \sum_{j=1}^m \D{(\mu_m^\alpha)_j}{\alpha}
    \ge \log \cZ_{m}(\alpha).
\]
Taking the supremum over \(\alpha\in\cP^+(A)\) and using \(\mu_m^\alpha\in\cN_m\) from \ref{cond:mu-alpha-in-cN} gives
\[
    \sup_{\mu\in\cN_m} \biggl\{ \Ex_\mu U_{m}(X) + \Ent(X) - \sum_{j=1}^m \Ent(X_j) \biggr\}
    \ge \sup_{\alpha\in\cP^+(A)} \log \cZ_{m}(\alpha)
    = F_m.
\]
The reverse inequality follows from \(\cN_m\subseteq\cE_m\) and \eqref{eq:abstract-Gibbs-variational-formula-Em}, which completes the proof.
\end{proof}

Using \Cref{abstract-Gibbs-variational-formula,entropy-deletion}, together with closure properties of {\cLor} laws in \Cref{Nm-closure}, we derive a monotonicity principle for \(F_m\) given the aforementioned properties.

\begin{proposition}\label{abstract-pressure-monotonicity}
Assume \ref{cond:mu-alpha-in-cN} and \ref{cond:energy-deletion}. Then for \(m\ge2\), \(mF_{m-1}\ge (m-2)F_m\).
\end{proposition}
\begin{proof}
Let \(\epsilon>0\) be arbitrary. Using \Cref{abstract-Gibbs-variational-formula}, choose \(\mu\in\cN_m\) such that
\begin{equation}\label{eq:F-epsilon-optimizer}
    F_m - \epsilon \le \Ex_{\mu} U_{m}(X) + \Ent(X) - \sum_{j=1}^m \Ent(X_j).
\end{equation}
For each \(j\in[m]\), \Cref{Nm-closure} implies that the law of \(X_{[m]\setminus j}\) belongs to \(\cN_{m-1}\). Thus, applying \Cref{abstract-Gibbs-variational-formula} with \(m\) replaced by \(m-1\) gives
\[
    F_{m-1} \ge \Ex_{\mu} U_{m-1}(X_{[m]\setminus j}) + \Ent(X_{[m]\setminus j}) - \sum_{k\neq j} \Ent(X_k).
\]
Summing this inequality over \(j\in[m]\) yields
\begin{align*}
    m F_{m-1}
    &\ge \Ex_{\mu} \sum_{j=1}^m U_{m-1}(X_{[m]\setminus j}) + \sum_{j=1}^m \Ent(X_{[m]\setminus j}) - (m-1) \sum_{j=1}^m \Ent(X_j).
\end{align*}
Using \ref{cond:energy-deletion} in this and applying the entropy inequality in \Cref{entropy-deletion} gives
\begin{align*}
    m F_{m-1}&\geq (m-2) \Ex_{\mu} U_{m}(X) + \sum_{j=1}^m \Ent(X_{[m]\setminus j}) - (m-1) \sum_{j=1}^m \Ent(X_j)
    \tag{by \ref{cond:energy-deletion}}
    \\&\ge (m-2) \Ex_{\mu} U_{m}(X) + \biggl( (m-2)\Ent(X) + \sum_{j=1}^m \Ent(X_j) \biggr) - (m-1) \sum_{j=1}^m \Ent(X_j)
    \tag{by \Cref{entropy-deletion}}
    \\&= (m-2) \biggl( \Ex_{\mu} U_{m}(X) + \Ent(X) - \sum_{j=1}^m \Ent(X_j) \biggr)
    \\&\ge (m-2) (F_m - \epsilon).
    \tag{by \eqref{eq:F-epsilon-optimizer}}
\end{align*}
Letting \(\epsilon\to0\) proves the claim.
\end{proof}

\section{Clique minimization for antiferromagnetic models}
\label{sec:antiferromagnetic}

We are ready to prove \Cref{antiferromagnetic-inhomo-clique-minimizing}.
Throughout this section, we fix an antiferromagnetic model $H$ on a finite nonempty spin set \(I\). Recall that \(Z_G\) and \(\Zi_G\) denote the vertex-homogeneous and vertex-inhomogeneous partition functions, respectively, i.e.,
\[
    \Zi_G(\blambda^{(v)}:v\in V(G))
    =
    \sum_{\bsigma\in I^{V(G)}}
    \prod_{uv\in E(G)} H(\sigma_u,\sigma_v)
    \prod_{v\in V(G)} \lambda^{(v)}_{\sigma_v},
\]
and $Z_G(\blambda)$ is the function $\Zi_G$ evaluated with $\blambda^{(v)}=\blambda$ for all $v\in V(G)$.
For \(\sigma\in I\), let \(\bfm_\sigma\) denote the row of \(H\) indexed by \(\sigma\). By continuity with respect to the fugacity vectors, it suffices to consider the case when every fugacity vector is entrywise positive.

We prove by induction on \(n=v(G)\) that \eqref{eq:inhomo-clique-minimizing} holds for every antiferromagnetic model \(H\) and every collection of positive fugacity vectors. The cases \(n=0\) and \(n=1\) are immediate. For the induction step, we may assume that \(n\ge2\) and \(G\) is connected, since both sides of \eqref{eq:inhomo-clique-minimizing} factor over connected components of \(G\). In particular, we may assume that \(G\) has no isolated vertices. If \(H\) is the zero matrix, both sides vanish. Otherwise, delete the zero rows and corresponding columns of \(H\) and restrict the fugacity vectors to the remaining spins. This leaves both sides unchanged and produces a nonempty antiferromagnetic model with no zero rows. Therefore, we may assume that \(H\) has no zero rows.



By \Cref{antiferromagnetic-irreducible}, \(H\) is irreducible. Hence, by the Perron--Frobenius theorem (\Cref{Perron--Frobenius}), \(H\) has an eigenvalue \(\rho>0\) with a corresponding entrywise-positive eigenvector \(\br\). Normalize \(\br\) so that \(\norm{\br}_2=1\). Since \(H\) is antiferromagnetic, \(\rho\) is its unique positive eigenvalue. Then for \(\epsilon>0\), \(H^{(\epsilon)} \defeq H+\epsilon \br\br^\top\) is entrywise positive. We claim that \(H^{(\epsilon)}\) is antiferromagnetic. Since \(H\) is symmetric, choose an orthonormal eigenbasis \(\cB\) for \(H\) containing \(\br\). Then for \(\bu\in\cB\setminus\{\br\}\), \(H^{(\epsilon)} \bu = H \bu\), while \(H^{(\epsilon)} \br = H \br + \epsilon \br = (\rho + \epsilon) \br\). Thus, the spectrum of \(H^{(\epsilon)}\) is the same as that of \(H\) except that the eigenvalue \(\rho\) is replaced by \(\rho + \epsilon\). Since \(\rho\) is the unique positive eigenvalue of \(H\), it follows that \(H^{(\epsilon)}\) is antiferromagnetic.
Both sides of \eqref{eq:inhomo-clique-minimizing} are continuous with respect to the entries of \(H\), so proving the inequality for \(H^{(\epsilon)}\) and then letting \(\epsilon\to0\) concludes its proof. Hence, in what follows in this section, we assume that \(H\) is entrywise positive.

\subsection{Inductive localization and the dual set}\label{sec:membership}

Next, we reduce the problem to a membership problem which depends only on a maximum degree vertex and its neighbors.
As explained in \Cref{sec:proof-strategy}, this part follows the argument in our previous work \cite{lee2026lower}.
We also note that the assumption that \(H\) is antiferromagnetic is not used in this subsection.

Fixing a vertex \(w\) of maximum degree \(\Delta\) and conditioning on its spin, we obtain a recurrence relation for the vertex-inhomogeneous partition functions (\Cref{inhomo-partition-function-recurrence}).
Applying the induction hypothesis and canceling from the desired inequality the common terms corresponding to vertices in \(V(G)\setminus(\{w\}\cup N(w))\), we reduce the problem to a local inequality on \(w\) and its neighbors, as given in \eqref{eq:goal-Phi-fraction}.
To eliminate the dependence on \(\blambda^{(w)}\) in this local inequality, we consider its equivalent form in \eqref{eq:goal-Phi-fraction-in-SDelta}, which is a membership problem for the set \(\cS_\Delta\), the dual set associated with the normalized clique partition function \((Z_{\Delta+1})^{1/(\Delta+1)}\).
Utilizing log-convexity of \(\cS_\Delta\) (\Cref{SDelta-log-convex}), we reduce this problem to a local membership problem (\Cref{SDelta-membership}), which we prove in the subsequent subsection.

We start by establishing a recurrence relation for the vertex-inhomogeneous partition functions, which generalizes the well-known identity $i(G) = i(G\setminus w) + i(G\setminus(N(w)\cup\{w\}))$ for independent set counts.

\begin{lemma}\label{inhomo-partition-function-recurrence}
Let \(\blambda^{(v)} \in (\RRp)^I\) for each \(v \in V(G)\). For \(w\in V(G)\),
\begin{equation}\label{eq:inhomo-partition-function-recurrence}
    \Zi_G(\blambda^{(v)} : v\in V(G)) = \sum_{\sigma\in I} \lambda^{(w)}_\sigma \Zi_{G\setminus w}(\bmu^{(v,\sigma)} : v\in V(G\setminus w)),
\end{equation}
where for \(v\in V(G\setminus w)\), \(\bmu^{(v,\sigma)}\in (\RRp)^I\) is defined as \(\blambda^{(v)}\) if \(v\notin N(w)\) and \(\bfm_\sigma\odot\blambda^{(v)}\) if \(v\in N(w)\).
\end{lemma}
\begin{proof}
By definition,
\[
    \Zi_G(\blambda^{(v)}:v\in V(G))
    = \sum_{\bfeta\colon V(G)\to I}
    \prod_{uv\in E(G)} H(\eta_u,\eta_v)
    \prod_{v\in V(G)} \lambda^{(v)}_{\eta_v}.
\]
We partition the sum according to the spin assigned to \(w\). Fix \(\sigma\in I\), and consider an assignment \(\bfeta\colon V(G)\to I\) satisfying \(\eta_w=\sigma\). Let \(\btau\) denote its restriction to \(V(G\setminus w)\). Then
\begin{align*}
    \prod_{uv\in E(G)}H(\eta_u,\eta_v)
    \prod_{v\in V(G)} \lambda^{(v)}_{\eta_v}
    &= \lambda^{(w)}_\sigma
    \Biggl(\, \prod_{uv\in E(G\setminus w)}H(\tau_u,\tau_v) \Biggr)
    \Biggl(\, \prod_{v\in N(w)}H(\sigma,\tau_v) \Biggr)
    \Biggl(\, \prod_{v\in V(G\setminus w)} \lambda^{(v)}_{\tau_v} \Biggr)
    \\&= \lambda^{(w)}_\sigma
    \Biggl(\, \prod_{uv\in E(G\setminus w)}H(\tau_u,\tau_v) \Biggr)
    \Biggl(\, \prod_{v\in V(G\setminus w)} \mu^{(v,\sigma)}_{\tau_v} \Biggr),
\end{align*}
because
\[
    \mu^{(v,\sigma)}_{\tau_v}
    = \begin{cases}
        \lambda^{(v)}_{\tau_v}, & v\notin N(w),\\
        H(\sigma,\tau_v)\lambda^{(v)}_{\tau_v}, & v\in N(w),
    \end{cases}
\]
and the second case is exactly
\[
    (\bfm_\sigma\odot\blambda^{(v)})_{\tau_v} = H(\sigma,\tau_v) \lambda^{(v)}_{\tau_v}.
\]
Therefore, summing first over \(\btau\colon V(G\setminus w)\to I\) and then over \(\sigma\in I\) gives the desired recurrence.
\end{proof}

Let \(w\in V(G)\) be a vertex of maximum degree \(\Delta\ge1\). Applying the induction hypothesis to each summand in \eqref{eq:inhomo-partition-function-recurrence} reduces the induction step to proving
\begin{align}\label{eq:induction}
     \sum_{\sigma\in I} \lambda^{(w)}_\sigma
    \prod_{v\in V(G\setminus w)} Z_{f_v+1}(\bmu^{(v,\sigma)})^{1/(f_v+1)}
    \ge \prod_{v\in V(G)} Z_{d_v+1}(\blambda^{(v)})^{1/(d_v+1)},
\end{align}
where for \(v\in V(G\setminus w)\), \(f_v\) is the degree of vertex \(v\) in \(G\setminus w\). If \(v\in V(G\setminus w)\setminus N(w)\), then \(f_v=d_v\) and \(\bmu^{(v,\sigma)} = \blambda^{(v)}\), so the corresponding factors on both sides cancel. If \(v\in N(w)\), then \(f_v=d_v-1\). Hence, \eqref{eq:induction} reduces to
\[
    \sum_{\sigma\in I} \lambda^{(w)}_\sigma
    \prod_{v\in N(w)} Z_{d_v}(\bmu^{(v,\sigma)})^{1/d_v}
    \ge \prod_{v\in \{w\}\cup N(w)} Z_{d_v+1}(\blambda^{(v)})^{1/(d_v+1)}.
\]
For \(d\ge1\), let \(\Phi_d(\blambda)\defeq Z_d(\blambda)^{1/d}\). By definition, \(\bmu^{(v,\sigma)} = \bfm_\sigma\odot\blambda^{(v)}\) for \(v\in N(w)\) and \(d_w=\Delta\), so the preceding inequality is equivalent to
\begin{equation}\label{eq:goal-Phi-fraction}
    \sum_{\sigma\in I} \lambda^{(w)}_\sigma
    \prod_{v\in N(w)} \frac{\Phi_{d_v}(\bfm_\sigma\odot\blambda^{(v)})}{\Phi_{d_v+1}(\blambda^{(v)})}
    \ge \Phi_{\Delta+1}(\blambda^{(w)}).
\end{equation}
Indeed, this generalizes~\eqref{eq:simplified_tangent} mentioned in the introduction. Define the \emph{dual set}
\[
    \cS_\Delta \defeq
    \{\bx\in(\RRp)^I :
    \forall \bz\in(\RRp)^I,\,
    \bz \cdot \bx \ge \Phi_{\Delta+1}(\bz)\}.
\]
Then \eqref{eq:goal-Phi-fraction} holds for every \(\blambda^{(w)}\in(\RRp)^I\) if and only if
\begin{equation}\label{eq:goal-Phi-fraction-in-SDelta}
    \Biggl(\,\prod_{v\in N(w)} \frac{\Phi_{d_v}(\bfm_\sigma \odot \blambda^{(v)})}{\Phi_{d_v+1}(\blambda^{(v)})} : \sigma\in I \Biggr)
    \in \cS_\Delta.
\end{equation}
Observe that the vector in \eqref{eq:goal-Phi-fraction-in-SDelta} can be interpreted as a coordinatewise product of \(\abs{N(w)}\) vectors, one for each neighbor of \(w\).
Thus, if \(\cS_\Delta\) is closed under coordinatewise geometric mean, then we can reduce \eqref{eq:goal-Phi-fraction-in-SDelta} to a separate membership statement for each neighbor of \(w\).
The next lemma establishes the required log-convexity.
The proof is essentially identical to that of Lemma~3.4 in \cite{lee2026lower} and uses no structural information about \(H\), but only the fact that \(Z_{\Delta+1}\) is a polynomial with nonnegative coefficients.
For \(\bx\in(\RRp)^I\), let \(\log\bx\defeq (\log x_\sigma : \sigma\in I)\).

\begin{lemma}\label{SDelta-log-convex}
\(\log\cS_\Delta \defeq \{\log \bx : \bx \in \cS_\Delta\}\) is convex.
\end{lemma}
\begin{proof}
The set \(\cS_\Delta\) is relatively closed in \((\RRp)^I\), and the coordinatewise logarithm is a homeomorphism from \((\RRp)^I\) to \(\RR^I\). Hence, \(\log\cS_\Delta\) is closed, so it suffices to prove midpoint convexity. Fix \(\bx,\by \in \cS_\Delta\), and set \(\bu\defeq ((x_\sigma y_\sigma)^{1/2} : \sigma\in I)\). We show \(\bu\in\cS_\Delta\), which amounts to showing that \(\bz\cdot\bu\ge\Phi_{\Delta+1}(\bz)\) for each \(\bz\in(\RRp)^I\). For \(\bz_1\defeq ((y_\sigma/x_\sigma)^{1/2} z_\sigma : \sigma\in I)\), the assumption \(\bx\in\cS_\Delta\) gives
\[
    \bz\cdot \bu = \bz_1\cdot \bx \ge \Phi_{\Delta+1}(\bz_1).
\]
Similarly, for \(\bz_2\defeq ((x_\sigma/y_\sigma)^{1/2} z_\sigma : \sigma\in I)\), the assumption \(\by\in\cS_\Delta\) gives
\[
    \bz\cdot \bu = \bz_2\cdot \by \ge \Phi_{\Delta+1}(\bz_2).
\]
Taking the geometric mean of these inequalities shows that it suffices to prove \(\Phi_{\Delta+1}(\bz_1) \Phi_{\Delta+1}(\bz_2) \ge \Phi_{\Delta+1}(\bz)^2\), which is equivalent to
\[
    Z_{\Delta+1}(\bz_1) Z_{\Delta+1}(\bz_2) \ge Z_{\Delta+1}(\bz)^2.
\]
Write \(Z_{\Delta+1}(\bz) = \sum_{\balpha} c_{\balpha} \bz^{\balpha}\) in its monomial expansion. Since \(c_{\balpha} \ge 0\) for every \(\balpha\), the Cauchy--Schwarz inequality gives
\[
    Z_{\Delta+1}(\bz_1) Z_{\Delta+1}(\bz_2)
    = \biggl( \sum_{\balpha} c_{\balpha} (\bz_1)^{\balpha} \biggr)
        \biggl( \sum_{\balpha} c_{\balpha} (\bz_2)^{\balpha} \biggr)
    \ge \biggl( \sum_{\balpha} c_{\balpha} ((\bz_1)^{\balpha} (\bz_2)^{\balpha})^{1/2} \biggr)^{\!2}
    = Z_{\Delta+1}(\bz)^2,
\]
where the last equality uses \((\bz_1)_\sigma (\bz_2)_\sigma = z_\sigma^2\) for all \(\sigma\in I\) so that \((\bz_1)^{\balpha} (\bz_2)^{\balpha} = (\bz^{\balpha})^2\) for all \(\balpha\in(\NN_0)^I\).
\end{proof}

By \Cref{SDelta-log-convex}, \Cref{eq:goal-Phi-fraction-in-SDelta} follows from the following.

\begin{lemma}\label{SDelta-membership}
For \(1\le d\le \Delta\) and \(\blambda\in(\RRp)^I\),
\[
    \biggl( \biggl(\frac{\Phi_d(\bfm_\sigma \odot \blambda)}{\Phi_{d+1}(\blambda)}\biggr)^{\!\Delta} : \sigma\in I \biggr)
    \in \cS_\Delta.
\]
\end{lemma}
Indeed, assume that \Cref{SDelta-membership} holds. Then, for each \(v\in N(w)\),
\[
    \bx_v \defeq \biggl( \biggl(\frac{\Phi_{d_v}(\bfm_\sigma \odot \blambda^{(v)})}{\Phi_{d_v+1}(\blambda^{(v)})}\biggr)^{\!\Delta} : \sigma\in I \biggr)
    \in \cS_\Delta.
\]
This is the subproblem associated with \(v\) described immediately before \Cref{SDelta-log-convex}.
Since \(\abs{N(w)}=\Delta\), \Cref{SDelta-log-convex} implies
\[
    \log \Biggl(\,\prod_{v\in N(w)} \frac{\Phi_{d_v}(\bfm_\sigma \odot \blambda^{(v)})}{\Phi_{d_v+1}(\blambda^{(v)})} : \sigma\in I \Biggr)
    = \sum_{v\in N(w)} \frac{1}{\Delta} \log \bx_v
    \in \log \cS_\Delta,
\]
which is equivalent to \eqref{eq:goal-Phi-fraction-in-SDelta}. Therefore, it remains to prove \Cref{SDelta-membership}.

\subsection{Resolution of the local membership problem}\label{sec:optimized-pressure}
For \(H=K_3^\circ\), the statement which corresponds to \Cref{SDelta-membership} appears as \cite[Lemma~3.3]{lee2026lower}.
The case \(\Delta=d\) of that statement is the tangent-plane inequality for the normalized clique partition function of clique size \(d+1\), while the case \(\Delta>d\) was proved by model-specific calculations and an iteration through consecutive degrees.
The same tangent-plane argument gives the starting point here.

For \(m\ge1\) and \(\bs,\bx\in(\RRp)^I\), define
\[
    D_{m,\bs}(\bx) \defeq \sum_{\sigma\in I} s_\sigma^{m-1} x_\sigma,
    \quad\text{and}\quad
    M_m(\bs) \defeq \sup_{\bx\in(\RRp)^I} \frac{Z_m(\bx)}{D_{m,\bs}(\bx)^m}.
\]
By definition, \(M_{\Delta+1}(\bs) \le1\) if and only if for all \(\bz\in(\RRp)^I\),
\[
    \Phi_{\Delta+1}(\bz)
    \le D_{\Delta+1,\bs}(\bz)
    = \sum_{\sigma\in I} s_\sigma^{\Delta} z_\sigma.
\]
Equivalently, \((s_\sigma^\Delta:\sigma\in I) \in \cS_\Delta\). 
Let 
\[
    A_{d,\sigma}(\bx) \defeq \frac{\Phi_d(\bfm_\sigma\odot\bx)}{\Phi_{d+1}(\bx)}.
\]
Thus, to prove \Cref{SDelta-membership}, it suffices to show the following:
\[
    \text{for all } \bx\in(\RRp)^I,
    \text{ if } \bs \defeq (A_{d,\sigma}(\bx) : \sigma\in I),
    \text{ then } M_{\Delta+1}(\bs)\le 1.
\]
Let us denote this property by $\mathsf{P}(d,\Delta)$, corresponding to neighbor degree \(d\) and clique size \(\Delta+1\).
The next lemma verifies $\mathsf{P}(m,m)$, which is the promised starting point derived from the tangent-plane argument.

\begin{lemma}\label{antiferromagnetic-Mm-bound}
For \(m\ge1\) and \(\bx\in(\RRp)^I\), let \(\bs\defeq (A_{m,\sigma}(\bx) : \sigma\in I)\). Then \(M_{m+1}(\bs) = 1\).
\end{lemma}
\begin{proof}
For \(\sigma\in I\), \eqref{eq:partition-function-derivative} gives
\begin{align*}
    \partial_\sigma \Phi_{m+1}(\bx)
    &= \frac{1}{m+1} Z_{m+1}(\bx)^{-m/(m+1)} \cdot \partial_\sigma Z_{m+1}(\bx)
    \\&= \frac{1}{m+1} Z_{m+1}(\bx)^{-m/(m+1)} \cdot (m+1) Z_m(\bfm_\sigma\odot\bx)
    \\&= \biggl( \frac{\Phi_m(\bfm_\sigma\odot\bx)}{\Phi_{m+1}(\bx)} \biggr)^{\!m}
    = s_\sigma^m.
\end{align*}
By \Cref{H-antiferromagnetic-Z-Lorentzian}, \(\Phi_{m+1}\) is concave on the positive orthant, so for all \(\bz\in(\RRp)^I\),
\begin{equation}\label{eq:Phi-concave-compare-D}
    \Phi_{m+1}(\bz)
    \le \Phi_{m+1}(\bx) + \sum_{\sigma\in I} \partial_\sigma \Phi_{m+1}(\bx) (z_\sigma-x_\sigma)
    = \sum_{\sigma\in I} \partial_\sigma \Phi_{m+1}(\bx) z_\sigma
    = \sum_{\sigma\in I} s_\sigma^m z_\sigma
    = D_{m+1,\bs}(\bz).
\end{equation}
Here, the first equality uses \(\sum_{\sigma\in I} \partial_\sigma \Phi_{m+1}(\bx) x_\sigma = \Phi_{m+1}(\bx)\), which follows from \Cref{euler-homogeneous}, since \(\Phi_{m+1}\) is homogeneous of degree 1. Therefore, \(M_{m+1}(\bs) \le 1\). Taking \(\bz=\bx\) gives equality in \eqref{eq:Phi-concave-compare-D}, so \(M_{m+1}(\bs)=1\).
\end{proof}

It remains to prove that $\mathsf{P}(d,d')$ implies $\mathsf{P}(d,d'+1)$ for $1\le d\le d' < \Delta$.
The key idea is to realize \(\log M_m(\bs)\) as the optimized pressure $F_m(\bs)$ of the abstract system introduced in \Cref{sec:abstract-optimized-pressure}, thereby allowing us to apply \Cref{abstract-pressure-monotonicity}.
The resulting monotonicity property of \(M_m(\bs)\) proves the desired implication from \(\mathsf{P}(d,d')\) to \(\mathsf{P}(d,d'+1)\).

Fix \(\bs\in(\RRp)^I\). In the abstract setting of \Cref{sec:abstract-optimized-pressure}, let the set of letters be \(A\defeq I\), let \(\Omega_m\defeq I^m\), and define the energy function by
\begin{align}\label{eq:energy}
    U_{m,\bs}(\bomega) \defeq \sum_{1\le i<j\le m} \log H(\omega_i,\omega_j) - (m-1) \sum_{j=1}^m \log s_{\omega_j},
\end{align}
which is well-defined because \(H\) and \(\bs\) are entrywise positive. Clearly, each \(\Omega_m\) is nonempty, \((\Omega_m : m\ge1)\) is permutation-invariant and deletion-compatible, and \(U_{m,\bs}\) is permutation-invariant. The corresponding partition function and the optimized pressure are denoted by
\begin{align}\label{eq:partition_and_pressure}
    \cZ_{m,\bs}(\alpha) \defeq \sum_{\bomega\in\Omega_m} \exp(U_{m,\bs}(\bomega)) \prod_{j=1}^m \alpha(\omega_j)
    \quad\text{and}\quad
    F_m(\bs) \defeq \sup_{\alpha\in\cP^+(A)} \log \cZ_{m,\bs}(\alpha),
\end{align}
respectively.
The pair-interaction terms in \(U_{m,\bs}(\bomega)\), which form the first sum in the definition, reproduce the weight of a spin assignment on \(K_m\).
The \(\bs\)-dependent terms introduce the normalization factors \(s_\sigma^{-(m-1)}\), chosen so that
\[
    \cZ_{m,\bs}(\alpha) = Z_m(\alpha(\sigma) s_\sigma^{-(m-1)} : \sigma\in I),
\]
thereby allowing us to compare \(F_m(\bs)\) with \(M_m(\bs)\).
This comparison is given in \Cref{F-eq-log-M}.
The coefficient \(m-1\) will also ensure that the family of energies satisfies the deletion identity \ref{cond:energy-deletion}.

\begin{lemma}\label{F-eq-log-M}
For \(m\ge1\) and \(\bs\in(\RRp)^I\), \(F_m(\bs) = \log M_m(\bs)\).
\end{lemma}
\begin{proof}
Fix \(m\ge1\) and \(\bs\in(\RRp)^I\). For \(\bx\in(\RRp)^I\), define \(\alpha_{\bx}\in\cP^+(A)\) by
\[
    \alpha_{\bx}(\sigma) \defeq \frac{s_\sigma^{m-1} x_\sigma}{D_{m,\bs}(\bx)},
    \qquad \sigma\in I.
\]
The map \(\bx\mapsto\alpha_{\bx}\) is a bijection from \(\{\bx\in(\RRp)^I : D_{m,\bs}(\bx)=1\}\) to \(\cP^+(A)\) with the inverse \(\alpha\mapsto(\alpha(\sigma) s_\sigma^{-(m-1)} : \sigma\in I)\).
Suppose \(D_{m,\bs}(\bx)=1\) and set \(\alpha=\alpha_{\bx}\). Then by definition,
\begin{equation}\label{eq:cZ-and-Z}
    \cZ_{m,\bs}(\alpha)
    = Z_m(\alpha(\sigma) s_\sigma^{-(m-1)} : \sigma\in I)
    = Z_m(\bx).
\end{equation}
To elaborate, summing
\[
    \exp(U_{m,\bs}(\bomega)) \prod_{j=1}^m \alpha(\omega_j) 
    = \prod_{1\le i<j\le m} H(\omega_i,\omega_j) \prod_{j=1}^m (\alpha(\omega_j)  s_{\omega_j}^{-(m-1)})
\]
over $\bomega=(\omega_j:j\in [m])\in I^m$ proves~\eqref{eq:cZ-and-Z}.
Using \eqref{eq:cZ-and-Z} together with the bijection \(\bx\mapsto\alpha_{\bx}\) gives
\[
    F_m(\bs)
    = \sup_{\alpha\in\cP^+(A)} \log \cZ_{m,\bs}(\alpha)
    = \sup_{\bx\in(\RRp)^I : D_{m,\bs}(\bx)=1} \log Z_m(\bx).
\]
Since both \(Z_m(\bx)\) and \(D_{m,\bs}(\bx)^m\) are homogeneous of degree \(m\),
\[
    \sup_{\bx\in(\RRp)^I : D_{m,\bs}(\bx)=1} \log Z_m(\bx)
    = \sup_{\bx\in(\RRp)^I} \log \biggl( \frac{Z_m(\bx)}{D_{m,\bs}(\bx)^m} \biggr)
    = \log M_m(\bs).
\]
Therefore, \(F_m(\bs) = \log M_m(\bs)\).
\end{proof}

Next, we show that the current setting satisfies the required properties, thereby allowing us to apply \Cref{abstract-pressure-monotonicity}.

\begin{lemma}\label{antiferromagnetic-abstract-conditions}
For every fixed \(\bs\in(\RRp)^I\), the partition function $\cZ_{m,\bs}(\alpha)$ in~\eqref{eq:partition_and_pressure} and the energy function in~\eqref{eq:energy} satisfy \ref{cond:mu-alpha-in-cN} and \ref{cond:energy-deletion}.
\end{lemma}
\begin{proof}
We first verify \ref{cond:mu-alpha-in-cN}. Fix \(m\ge1\) and \(\alpha\in\cP^+(A)\). Recall that \(\mu_m^\alpha\), defined by
\[
    \mu_m^\alpha(\bomega) \defeq \frac{1}{\cZ_{m,\bs}(\alpha)} \exp(U_{m,\bs}(\bomega)) \prod_{j=1}^m \alpha(\omega_j),
\]
is an exchangeable law on \(\Omega_m\). For simplicity, write \(\mu\defeq\mu_m^\alpha\). We claim that the count-generating polynomial of \(\mu\) is
\begin{equation}\label{eq:antiferromagnetic-P1-count-generating-poly}
    h_\mu(\bz) = \frac{Z_m(\alpha(\sigma) s_\sigma^{-(m-1)} z_\sigma : \sigma\in I)}{\cZ_{m,\bs}(\alpha)}.
\end{equation}
Indeed, by definition,
\[
    h_\mu(\bz)
    = \sum_{\bomega\in\Omega_m} \mu(\bomega) \bz^{T(\bomega)},
    \qquad
    \bz^{T(\bomega)} = \prod_{j=1}^m z_{\omega_j}.
\]
For each \(\bomega\in\Omega_m\), \(\cZ_{m,\bs}(\alpha)\mu(\bomega)\bz^{T(\bomega)}\) is equal to
\[
    \exp(U_{m,\bs}(\bomega)) \prod_{j=1}^m (\alpha(\omega_j) z_{\omega_j})
    = \prod_{1\le i<j\le m} H(\omega_i,\omega_j) \prod_{j=1}^m (\alpha(\omega_j) z_{\omega_j} s_{\omega_j}^{-(m-1)}).
\]
Summing over \(\bomega\in\Omega_m\) gives exactly
\[
    Z_m(\alpha(\sigma) s_\sigma^{-(m-1)} z_\sigma : \sigma\in I),
\]
which proves \eqref{eq:antiferromagnetic-P1-count-generating-poly}. By \Cref{H-antiferromagnetic-Z-Lorentzian}, \(Z_m\) is Lorentzian, and \Cref{Lorentzian-closure} shows that \(h_\mu\) is also Lorentzian. Therefore, \(\mu\in\cN_m\).

Next, we verify \ref{cond:energy-deletion}. Fix \(m\ge2\) and \(\bomega\in\Omega_m\). The energy \(U_{m,\bs}\) is a sum of pair-interaction terms and linear terms. Pair-interaction terms satisfy
\[
    \sum_{k=1}^m \sum_{\substack{1\le i<j\le m \\ i,j\neq k}} \log H(\omega_i,\omega_j)
    = (m-2) \sum_{1\le i<j\le m} \log H(\omega_i,\omega_j),
\]
because each pair \((i,j)\) survives exactly \(m-2\) of the \(m\) single-site deletions. Linear terms satisfy
\[
    \sum_{k=1}^m \biggl( -(m-2) \sum_{i\in [m]\setminus \{k\}} \log s_{\omega_i} \biggr)
    = -(m-2)(m-1) \sum_{i=1}^m \log s_{\omega_i},
\]
because each site \(i\) survives exactly \(m-1\) deletions. Combining these gives
\[
    \sum_{k=1}^m U_{m-1,\bs}(\bomega_{[m]\setminus k}) = (m-2) U_{m,\bs}(\bomega),
\]
which is exactly \ref{cond:energy-deletion} for the energy function $U_{m,\bs}$.
\end{proof}

\Cref{antiferromagnetic-abstract-conditions,abstract-pressure-monotonicity} imply that, for \(m\ge3\) and \(\bs\in(\RRp)^I\),
\begin{equation}\label{eq:antiferromagnetic-pressure-monotonicity}
    m F_{m-1}(\bs) \ge (m-2) F_m(\bs).
\end{equation}
Finally, we prove \Cref{SDelta-membership}.

\begin{proof}[Proof of \Cref{SDelta-membership}]
Let \(\blambda\in(\RRp)^I\) and \(\bs\defeq (A_{d,\sigma}(\blambda) : \sigma\in I)\).
\Cref{antiferromagnetic-Mm-bound} implies \(M_{d+1}(\bs) = 1\), which gives \(\mathsf{P}(d,d)\).
This immediately gives \(\mathsf{P}(d,\Delta)\) if \(\Delta=d\), so assume \(1\le d<\Delta\).
Suppose \(1\le d\le d'<\Delta\) and \(\mathsf{P}(d,d')\) holds.
By \Cref{F-eq-log-M}, \(F_{d'+1}(\bs) = \log M_{d'+1}(\bs) \le 0\), so \eqref{eq:antiferromagnetic-pressure-monotonicity} with \(m=d'+2\) implies \(F_{d'+2}(\bs)\le 0\), where we use the fact that the coefficient \(m-2=d'\) is nonzero.
Applying \Cref{F-eq-log-M} again gives \(M_{d'+2}(\bs) \le 1\), hence \(\mathsf{P}(d,d'+1)\) holds.
Therefore, \(\mathsf{P}(d,\Delta)\) holds, which completes the proof.
\end{proof}

By the reduction in \Cref{sec:membership}, \Cref{SDelta-membership} implies \eqref{eq:goal-Phi-fraction} and therefore completes the induction step. This proves \Cref{antiferromagnetic-inhomo-clique-minimizing}.

\section{Concluding remarks}\label{sec:concluding-remark}

A natural next problem is to determine the precise scope of the clique-minimizing phenomenon.

\begin{problem}
Characterize all clique-minimizing models.
\end{problem}

As mentioned in the introduction, the class of clique-minimizing models strictly contains the class of antiferromagnetic models. Let \(H_1\) and \(H_2\) be antiferromagnetic models on spin sets \(I_1\) and \(I_2\), respectively, and let \(M\defeq H_1\otimes H_2\), where \(\otimes\) denotes the tensor product. That is, \(M\) has the spin set \(I_1\times I_2\) and
\[
    M((\sigma_1,\sigma_2), (\tau_1,\tau_2)) \defeq H_1(\sigma_1,\tau_1) H_2(\sigma_2,\tau_2),
    \qquad \sigma_1,\tau_1\in I_1,\ \sigma_2,\tau_2\in I_2.
\]
The eigenvalues of \(M\) are \(\{\eta_i\rho_j : i\in I_1,\, j\in I_2\}\), where \(\{\eta_i : i\in I_1\}\) and \(\{\rho_j : j\in I_2\}\) are the multisets of eigenvalues of \(H_1\) and \(H_2\), respectively. Thus, if \(H_1\) and \(H_2\) have negative eigenvalues, then \(M\) is not antiferromagnetic.
On the other hand, for every graph \(G\),
\begin{align*}
    \hom(G,M) = \hom(G,H_1) \hom(G,H_2)
    &\ge \prod_{v\in V(G)} \hom(K_{d_v+1},H_1)^{1/(d_v+1)}
    \prod_{v\in V(G)} \hom(K_{d_v+1},H_2)^{1/(d_v+1)}
    \\&= \prod_{v\in V(G)} \hom(K_{d_v+1},M)^{1/(d_v+1)},
\end{align*}
so \(M\) is clique-minimizing. For example, \(K_2^\circ \otimes K_2\) is clique-minimizing but not antiferromagnetic.

Another way to construct examples is to consider models \(H\) such that \(\hom(K_r,H)=0\) for some \(r\ge3\). In this case, it suffices to verify \eqref{eq:clique-minimizing} for every \(G\) with maximum degree at most \(r-2\). Indeed, if \(d_v\ge r-1\) for some \(v\in V(G)\), then \(\hom(K_{d_v+1},H)=0\), so the right-hand side of \eqref{eq:clique-minimizing} vanishes. Both sides of \eqref{eq:clique-minimizing} factor over connected components of \(G\), so one may restrict further to connected \(G\).
When \(r=3\), every connected graph $G$ with maximum degree at most \(1\) is \(K_1\) or \(K_2\), and equality holds in both cases, so \eqref{eq:clique-minimizing} holds for any bipartite graph $H$, not necessarily antiferromagnetic.
When \(r=4\), it remains to check connected graphs \(G\) with \(\Delta(G)\le2\), which are paths and cycles. One may check that the model \(H\) given by a triangle with one pendant edge is clique-minimizing but not antiferromagnetic.

\medskip

It would also be interesting to understand equality and stability in the vertex-inhomogeneous inequality \eqref{eq:inhomo-clique-minimizing}.
As mentioned, disjoint unions of cliques with the fugacity vectors satisfying \(\blambda^{(u)} = \blambda^{(v)}\) for each \(uv\in E(G)\) attain equality. For a fixed antiferromagnetic model, one may ask whether these are the only extremizers under suitable nondegeneracy assumptions and whether near equality forces the source graph $G$ to be close to a disjoint union of cliques.

We note that it was conjectured in \cite[Conjecture~7.1]{lee2025counting} that, for a connected graph \(G\), the vertex-homogeneous partition function \(Z_G\) of \(H\) is Lorentzian for every antiferromagnetic model \(H\) if and only if \(G\) is a clique. Thus, this conjecture predicts that the Lorentzian property established in \Cref{H-antiferromagnetic-Z-Lorentzian} characterizes cliques among connected source graphs.

\paragraph{Acknowledgments}
This research was partly carried out during the second author's visit to Caltech, hosted by David Conlon and supported by the BK21 FOUR Support Program for Outstanding Graduate Students' International Joint Training. The second author would like to thank them for their hospitality and support.

\paragraph{Use of large language models}
Given the rapid development of AI and its growing application in mathematical research, we believe that transparent disclosure of AI use should be standard, as it would help reduce the temporary confusion surrounding the evaluation of human contributions, AI-assisted work, and the broader possibilities created by collaboration between researchers and AI systems.

Mathematically, we used large language models (LLMs), particularly GPT-5.5, as interactive tools during the development of the proof of \Cref{antiferromagnetic-inhomo-clique-minimizing}. 
Our use of these tools in the discovery process proceeded in three stages.
First, while we were studying the antiferromagnetic Ising case, i.e., the Davies--LeBlanc conjecture, GPT-5.5 Pro Extended produced a solution of a reduced form of the local membership problem.
This solution invoked, without proof, a version of the relative entropy contraction.
It then also produced a proof of the claimed entropy inequality, though we made extra queries to relate the argument to~\cite{anari2021logconcave,anari2022entropic}.
Second, we iteratively used GPT-5.5 Thinking Extended to investigate technical details and possible generalizations, leading first to proofs for antiferromagnetic 2-spin models and for \(K_q^\circ\) with the loop weight restricted to a certain range.
Comparing the resulting model-specific arguments helped us isolate their common structure and formulate the entropy inequalities and the argument based on optimized pressure in an abstract setting.
These arguments initially required the antiferromagnetic model to have unit off-diagonal entries or at most two spins.
Finally, with the assistance of GPT-5.5 Pro Extended, we extended this strategy to general antiferromagnetic models in a manner closely paralleling the earlier argument for models with unit off-diagonal entries.

We also used GPT-5.6 Sol Pro and Codex to polish the writing, improve the exposition, and support literature review, with all suggestions reviewed and refined by the authors.
The authors take full responsibility for the correctness of the arguments, citations, and the exposition of the paper.

\printbibliography

\end{document}